\documentclass[reqno]{amsart}
\usepackage{amssymb}
\usepackage[usenames, dvipsnames]{color}
\usepackage{pxfonts}
\usepackage{enumerate}
\usepackage{amsmath}
\usepackage[driverfallback=dvipdfm]{hyperref}{\tiny}
\usepackage{verbatim}

\allowdisplaybreaks[4]

\numberwithin{equation}{section}

\newtheorem{theorem}{Theorem}[section]
\newtheorem{corollary}[theorem]{Corollary}
\newtheorem{lemma}[theorem]{Lemma}
\newtheorem{prop}[theorem]{Proposition}

\theoremstyle{definitio`n}
\newtheorem{remark}[theorem]{Remark}

\theoremstyle{definition}
\newtheorem{definition}[theorem]{Definition}

\theoremstyle{definition}
\newtheorem{assumption}[theorem]{Assumption}

\theoremstyle{definition}

\makeatletter
\def\dashint{\operatorname%
{\,\,\text{\bf-}\kern-.98em\DOTSI\intop\ilimits@\!\!}}
\makeatother

\def\\det{\text{det}}

\def\.5{\frac{1}{2}}

\def\cD{\mathcal{D}}

\def\cQ{\mathcal{Q}}

\newcommand{\RN}[1]{%
  \textup{\uppercase\expandafter{\romannumeral#1}}%
}

\newcommand{\Div}{\operatorname{div}}
\newcommand{\dist}{\text{dist}}

\newcounter{marnote}

\begin{document}


\title[Gradient estimates for divergence form parabolic systems]{Gradient estimates for divergence form parabolic systems}

\author[H. Dong]{Hongjie Dong}
\address[H. Dong]{Division of Applied Mathematics, Brown University, 182 George Street, Providence, RI 02912, USA}
\email{Hongjie\_Dong@brown.edu }
\thanks{H. Dong was partially supported by the NSF under agreement DMS-1600593.}

\author[L. Xu]{Longjuan Xu}
\address[L. Xu]{Department of Mathematics, Yonsei University, 50 Yonsei-Ro, Seodaemun-gu, Seoul 03722, Republic of Korea.}
\email{longjuanxu@yonsei.ac.kr}

\begin{abstract}
We consider divergence form, second-order strongly parabolic systems in a cylindrical domain with a finite number of subdomains under the assumption that the interfacial boundaries are  $C^{1,\text{Dini}}$ and $C^{\gamma_{0}}$ in the spatial variables and the time variable, respectively. Gradient estimates and piecewise $C^{1/2,1}$-regularity are established when the leading coefficients and data are assumed to be of piecewise Dini mean oscillation or piecewise H\"{o}lder continuous. Our results improve the previous results in \cite{ll,fknn} to a large extent. We also prove a global weak type-$(1,1)$ estimate with respect to $A_{1}$ Muckenhoupt weights for the parabolic systems with leading coefficients which satisfy a stronger assumption. As a byproduct, we give a proof of optimal regularity of weak solutions to parabolic transmission problems with $C^{1,\mu}$ or $C^{1,\text{Dini}}$ interfaces. This gives an extension of a recent result in \cite{css} to parabolic systems.
\end{abstract}

\maketitle

\section{Introduction}

We are concerned with second-order parabolic systems in divergence form arising from composite materials. We are interested in obtaining the gradient estimates for such systems when the domain can be decomposed into a finite number of time-dependent subdomains with coefficients and data which are of piecewise Dini mean oscillation. See the specific definitions given in the next section. These estimates will be shown to be independent of the distance between subdomains. Such a problem also appears in the study of the evolution of fronts in fluid dynamics, where the interfacial boundaries are typically time-dependent. See, for instance, \cite{GG18}.

The well-known theory of De Giorgi-Nash-Moser states that weak solutions for divergence form second-order elliptic and parabolic equations are H\"{o}lder continuous when the leading coefficients are bounded and measurable. On the other hand, we recall that the examples in \cite{meyer,ps} reveal that  solutions to second-order elliptic and parabolic equations with bounded and measurable coefficients are in general not Lipschitz continuous. A natural question is that what is the minimal regularity assumption of the coefficients for the $C^{1}$ or Lipschitz regularity of weak solutions. See \cite{a,b} for results in this direction. In \cite{l}, Li proved $C^{1}$-regularity of solutions to divergence form elliptic systems
$$
D_{\alpha}(A^{\alpha\beta}D_{\beta}u)=0,
$$
provided that the modulus of the continuity of coefficients in the $L^{\infty}$ sense satisfies the Dini condition. This result was extended in \cite{dk} to non-homogeneous equations
\begin{equation}\label{system laminae}
D_{\alpha}(A^{\alpha\beta}D_{\beta}u)=\Div g,
\end{equation}
where the coefficients and data are only assumed to be of Dini mean oscillation. See also \cite{dek} for the corresponding boundary estimates.
Recently, Dong, Escauriaza, and Kim \cite{dek2} considered parabolic equations in divergence form with zero Dirichlet boundary conditions and showed that weak solutions are continuously differentiable in the space variables and $C^{1/2}$ in the time variable up to the boundary when the leading coefficients have Dini mean oscillation with respect to the space variables and the lower-order coefficients satisfy certain conditions.

There are also many works in the literature concerning the case when the domain contains subdomains and the coefficients are piecewise regular. See, for instance, \cite{basl,bv,xb}.
Chipot, Kinderlehrer, and  Vergara-Caffarelli \cite{ckvc} showed that any weak solution $u$ of \eqref{system laminae} is locally Lipschitz if $A^{\alpha\beta}$ are piecewise constants and $g\in C^{k}$ for $k\geq[d/2]$, when the domain consists of a finite number of linearly elastic, homogeneous, and parallel laminae. Li and Vogelius \cite{lv} considered scalar elliptic equations
$$
D_{\alpha}(a^{\alpha\beta}D_{\beta}u)=\Div g+f,
$$
where the matrix $(a^{\alpha\beta})$ and data are assumed to be $C^{\delta}$ up to the boundary in each subdomain with $C^{1,\mu}, 0<\mu\leq1,$ boundary, but may have jump discontinuities across the boundaries of the subdomains. The authors derived global Lipschitz and piecewise $C^{1,\delta'}$ estimates of the solution $u$ for any $\delta'\in (1,\min\{\delta,\frac{\mu}{d(\mu+1)}\}]$, with the estimates independent of the distance between subdomains. Li and Nirenberg \cite{ln} later extended their results to elliptic systems under the same conditions when $\delta'$ is in a larger range $(0,\min\{\delta,\frac{\mu}{2(\mu+1)}\}]$. Recently, Dong and Xu \cite{dx} improved the regularity of $u$ by further extending the range of $\delta'$ from $(0, \min\{\delta,\frac{\mu}{2(\mu+1)}\}]$ to $(0,\min\{\delta,\frac{\mu}{\mu+1}\}]$, which seems to be sharp. The proof is based on a weak type-$(1,1)$ estimate and Campanato's method, which are different from the $L^{2}$-estimates used in \cite{ln,lv}.

Parabolic equations have also been studied in this setting. In \cite{ll}, Li and Li extended the interior estimates in \cite{ln} to parabolic systems with coefficients which are piecewise H\"{o}lder continuous in the space variables and smooth in the time variable, when the subdomains are cylindrical. See also \cite{fknn}, where the coefficients are independent of the time variable and the subdomains are also assumed to be cylindrical. We would like to mention that in \cite{d} optimal regularity of weak solutions was obtained when the coefficients and data are Dini continuous in the time variable and all but one spatial variable.

The current paper is a natural extension of \cite{dx} from the elliptic case to the parabolic case.
We substantially  improve the results in the aforementioned papers \cite{ll,fknn} in the following two aspects. First, we allow the subdomains to be non-cylindrical, and the interfacial boundaries to be $C^{1,\text{Dini}}$ in the spatial variables and $C^{\gamma_{0}}$ in the time variable, where $\gamma_{0}>1/2$. Second, we relax the regularity assumption on the leading coefficients, particularly in the time variable. We show in Theorem \ref{thm1} that $\mathcal{H}_{p}^{1}$ $(1<p<\infty)$ weak solutions to parabolic systems in divergence form are Lipschitz in all spatial variables and piecewise $C^{1/2,1}$ when the leading coefficients and data are of piecewise Dini mean oscillation and the lower-order coefficients are bounded. Besides, we obtain the local $L_{p}$-estimate for $\mathcal{H}_{1}^{1}$ weak solutions in Corollary \ref{coro loc}  by adapting the idea in  \cite{a,b}, and thus the results in  Theorem \ref{thm1} also hold for these solutions. When the leading coefficients and data have piecewise H\"{o}lder regularity, we prove the piecewise $C^{1,\delta'}$-regularity in the spatial variables and $C^{(1+\delta')/2}$-regularity in the time variable of weak solutions in Theorem \ref{thm holder}, where $\delta'$ is in the optimal range $(1,\min\{\delta,\frac{\mu}{\mu+1},2\gamma_{0}-1\}]$.
	
Our arguments in proving Theorems \ref{thm1} and \ref{thm holder} are different from those in \cite{fknn,ll,ln,lv}. The proofs below are based on Campanato's method, which was also used recently in \cite{dx}. 
The key point is to show the mean oscillation of $Du$ in balls or cylinders vanishes in a certain order as the radii of the  balls or cylinders go to zero. However, this method cannot be employed directly because $Du$ is discontinuous across the interfacial boundaries and we only  impose the assumption on the $L_{1}$-mean oscillation of the coefficients and data, so that the usual argument based on $L_{p}$ $(p>1)$ estimates does not work here. To overcome these difficulties, we first fix the coordinate system and derive weak type-$(1,1)$ estimates by using a duality argument. We then establish some interior H\"{o}lder regularity of $D_{x'}u$ and $\bar{U}:=\bar{A}^{d\beta}D_{\beta}u$ for parabolic systems with coefficients depending on one variable, say, $x^{d}$.  The desired results in Theorem \ref{thm1} are proved by adapting Campanato's approach in the $L_{p}$ setting for some $p\in(0,1)$.  The proof of the $C^{(1+\delta')/2}$-H\"{o}lder continuity in $t$ of weak solutions in Theorem \ref{thm holder} is more involved. We prove a weak type-$(1,1)$ estimate and apply Campanato's idea to $u$ itself instead of its first derivatives. For this, we introduce a set consisting of functions in $x$ which are linear in $x'$ and prove Lemma \ref{lem3.13} which plays a key role in estimating the difference between $u$ and its approximations in the $L_{q}$-mean sense, $q\in(0,1)$. Compared to \cite{dx}, this is new and can be considered as the main contribution of the current paper.

As a byproduct, in Theorem \ref{trans holder} we prove the existence, uniqueness, and $C^{1,\mu}$ regularity of weak solutions to transmission problems with $C^{1,\mu}$ interfaces in the parabolic setting, which is an extension of a recent result in \cite{css}. We also consider a more general case when the interfaces are $C^{1,\text{Dini}}$. See Theorem \ref{trans dini}. For the proof, we adapt an idea in \cite{dong} by solving certain auxiliary equations in subdomains with conormal boundary data and then reducing the transmission problem to a parabolic equation with piecewise H\"older (or Dini) inhomogeneous terms. In contrast to the elliptic case, here we cannot treat the derivatives of solutions to the auxiliary equations as inhomogeneous terms because their time derivatives are in Sobolev spaces of negative order. In this paper, we modify the argument in \cite{dong} by considering the difference of $u$ and these auxiliary solutions. As such, we need to extend these solutions to the whole domain, which is achieved by a partition of unity argument applied to each subdomain together with a flattening-reflection technique.

Throughout this paper, unless otherwise stated, $N$ denotes a constant, whose value may vary from line to line and independent of the distance between subdomains. We call it a {\it{universal constant}}.

The rest of this paper is organized as follows. We formulate the problem and state our main results, Theorems \ref{thm1} and \ref{thm holder}, in Section \ref{problem results}. In Section \ref{preliminaries}, we introduce some notation, definitions, and auxiliary lemmas used in this paper. The main result Theorem \ref{thm1} under the assumptions that the leading coefficients and data are of piecewise Dini mean oscillation is proved in Section \ref{proof thm1}, where we also give the proof of Corollary \ref{coro loc}. We prove Theorem \ref{thm holder} in Section \ref{section thm holder}. Section \ref{sec thm3} is devoted to a global weak type-$(1,1)$ estimate with respect to $A_{1}$ Muckenhoupt weights for solutions to parabolic systems. In Section \ref{transmission prb}, we state and prove Theorems \ref{trans holder} and \ref{trans dini} by adapting the method in \cite{dong}. In the Appendix, we prove a weighted $\mathcal{H}_{p,w}^{1}$-solvability and estimate for divergence form parabolic systems in nonsmooth domains with partially VMO coefficients.

\section{Problem formulation and main results}\label{problem results}
\subsection{Problem formulation}\label{prob formulation}

In this paper, we aim to establish gradient estimates for strongly parabolic systems in divergence form
\begin{equation}\label{systems}
\mathcal{P}u:=-u_{t}+D_{\alpha}(A^{\alpha\beta}D_{\beta}u+B^{\alpha}u)+\hat{B}^{\alpha}D_{\alpha}u+Cu=\Div g+f
\end{equation}
in a cylindrical domain $\cQ:=(-T,0)\times\cD$, where $T\in(0,\infty)$ and
$\cD$ is a bounded domain in $\mathbb R^{d}$. We assume that $\cQ$ contains $M$ disjoint time-dependent subdomains $\cQ_{j},j=1,\ldots,M,$ and the interfacial boundaries are $C^{1,\text{Dini}}$ in the spatial variables and $C^{\gamma_{0}}$ in the time variable, where $\gamma_{0}>1/2$. See the details in Definition \ref{def Dini}. We also assume that any point $(t,x)\in\cQ$ belongs to the boundaries of at most two of the $\cQ_{j}$'s.
Moreover, the Einstein summation convention over repeated indices are assumed throughout this paper. Here
$$
u=(u^{1},\ldots,u^{n})^{\top},\quad g_{\alpha}=(g_{\alpha}^{1},\ldots,g_{\alpha}^{n})^{\top},\quad f=(f^{1},\ldots,f^{n})^{\top}
$$
are (column) vector-valued functions,
$A^{\alpha\beta}, B^{\alpha}, \hat{B}^{\alpha}$ (often denoted by $A, B, \hat{B}$ for abbreviation), and $C$ are $n\times n$ matrices, which are bounded by a positive constant $\Lambda$. The leading coefficients matrices $A^{\alpha\beta}$ satisfy the strong parabolicity condition: there exists a number $\nu>0$ such that for any $\xi=(\xi_{\alpha}^{i})\in\mathbb R^{n\times d}$,
$$\nu|\xi|^{2}\leq A_{ij}^{\alpha\beta}\xi_{\alpha}^{i}\xi_{\beta}^{j},\quad|A^{\alpha\beta}|\leq\nu^{-1}.$$

To localize the problem, we slightly abuse the notation by taking $\cQ$ to be a unit cylinder $Q_{1}^{-}:=(-1,0)\times B_{1}$ and $z_{0}=(t_{0},x_{0})\in (-9/16,0)\times B_{3/4}$. The domain is fixed as follows. By suitable rotation and scaling, we may suppose that a finite number of subdomains lie in $Q_{1}^{-}$ and that they can be represented by
$$x^{d}=h_{j}(t,x'),\quad\forall~t\in(-1,0),~x'\in B'_{1},~j=1,\ldots,l<M,$$
where
\begin{equation}\label{eq10.42}
-1<h_{1}(t,x')<\dots<h_{l}(t,x')<1,
\end{equation}
$h_{j}(t,\cdot)\in C^{1,\text{Dini}}(B'_{1})$, and $h_{j}(\cdot,x')\in C^{\gamma_{0}}(-1,0)$, where $\gamma_{0}>1/2$. Set $h_{0}(t,x')=-1$ and $h_{l+1}(t,x')=1$. Then we have $l+1$ regions:
$$\cQ_{j}:=\{(t,x)\in \cQ: h_{j-1}(t,x')<x^{d}<h_{j}(t,x')\},\quad1\leq j\leq l+1.$$ We may suppose that there exists some $\cQ_{j_{0}}$, such that $(t_{0},x_{0})\in \Big((-9/16,0)\times B_{3/4}\Big)\cap \cQ_{j_{0}}$ and the closest point on $\partial_p \cQ_{j_{0}}\cap\{t=t_0\}$ to $(t_{0},x_{0})$ is $(t_{0},x'_{0},h_{j_{0}}(t_{0},x'_{0}))$, and $\nabla_{x'}h_{j_{0}}(t_{0},x'_{0})=0'$.
We introduce the $l+1$ ``strips"
$$\Omega_{j}:=\{(t,x)\in \cQ: h_{j-1}(t_{0},x'_{0})<x^{d}<h_{j}(t_{0},x'_{0})\},\quad1\leq j\leq l+1.$$

Denote by $\mathcal{A}$ the set of piecewise constant functions in each $\cQ_{j}$. We then further assume that $A$ is of piecewise Dini mean oscillation in $\cQ$, that is,
\begin{align}\label{def omega A}
\omega_{A}(r):=\sup_{z_{0}\in \cQ}\inf_{\hat{A}\in\mathcal{A}}\fint_{Q_{r}^{-}(z_{0})}|A(z)-\hat{A}|\ dz
\end{align}
satisfies the Dini condition, where $Q_{r}^{-}(z_{0}):=(t_0-r^2,t)\times B_{r}(x_{0})\subset \cQ$. The reader can refer to Definition \ref{piece Dini} about the Dini condition. For $\varepsilon>0$ small, we set
$$\cD_{\varepsilon}:=\{x\in \cD: \mbox{dist}(x,\partial \cD)>\varepsilon\}.$$

\subsection{Main results}
We state the main results of this paper.
\begin{theorem}\label{thm1}
Let $\cQ$ be defined as above. Let $\varepsilon\in (0,1)$, $p\in(1,\infty)$, and $\gamma\in(0,1)$. Assume that $A$, $B$, and $g$ are of piecewise Dini mean oscillation in $\cQ$, and $f, g\in L_{\infty}(\cQ)$. If $u\in \mathcal{H}_{p}^{1}(\cQ)$ is a weak solution to \eqref{systems} in $\cQ$, then $u\in C^{1/2,1}(\overline{\cQ_{j}}\cap((-T+\varepsilon,0)\times\cD_{\varepsilon}))$, $j=1,\ldots,M$.
Moreover, for any fixed $z_{0}\in (-T+\varepsilon,0)\times\cD_{\varepsilon}$, there exists a coordinate system associated with $z_{0}$, such that for all $z\in (-T+\varepsilon,0)\times\cD_{\varepsilon}$, we have
\begin{align*}
&|(D_{x'}u(z_{0}),U(z_{0}))-(D_{x'}u(z),U(z))|\nonumber\\
&\leq N\int_{0}^{|z_{0}-z|_{p}}\frac{\tilde{\omega}_{g}(s)}{s}\ ds+N|z_{0}-z|_{p}^{\gamma}\left(\|Du\|_{L_{1}(\cQ)}+\|g\|_{L_{\infty}(\cQ)}
+\|f\|_{L_{\infty}(\cQ)}+\|u\|_{L_{p}(\cQ)}\right)\nonumber\\
&\ + N\int_{0}^{|z_{0}-z|_{p}}\frac{\tilde{\omega}_{A}(s)}{s}\ ds\cdot\left(\|Du\|_{L_{1}(\cQ)}+\int_{0}^{1}\frac{\tilde{\omega}_{g}(s)}{s}\ ds+\|g\|_{L_{\infty}(\cQ)}+\|f\|_{L_{\infty}(\cQ)}+\|u\|_{L_{p}(\cQ)}\right),
\end{align*}
where $U=A^{d\beta}D_{\beta}u+B^{d}u-g_{d}$, $|\cdot|_{p}$ is the parabolic distance defined in \eqref{para dis}, $N$ depends on $n,d,M,p,\Lambda,\nu,\varepsilon,\omega_{B}$, and the $C^{1,\text{Dini}}$ and $C^{\gamma_{0}}$ characteristics of $\cQ_{j}$ with respect to $x$ and $t$, respectively,  and $\tilde\omega_{\bullet}(t)$ is a Dini function derived from $\omega_{\bullet}(t)$. See \eqref{tilde phi}.
\end{theorem}

By using a duality argument and Theorem \ref{thm1}, we obtain the following result.
\begin{corollary}\label{coro loc}
Under the same conditions as in Theorem \ref{thm1}, if $u\in \mathcal{H}_{1}^{1}(\cQ)$ is a weak solution to \eqref{systems} in $\cQ$, then $u\in \mathcal{H}_{p,\text{loc}}^{1}(\cQ)$ for some $p\in(1,\infty)$ and for any $\cQ'\subset\subset\cQ$,
\begin{align*}
\|u\|_{\mathcal{H}_{p}^{1}(\cQ')}\leq N\left(\|g\|_{L_{\infty}(\cQ)}+\|f\|_{L_{\infty}(\cQ)}+\|u\|_{\mathcal{H}_{1}^{1}(\cQ)}\right).
\end{align*}
Furthermore, the conclusion of Theorem \ref{thm1} still holds true.
\end{corollary}

The next theorem shows that if we impose piecewise H\"{o}lder regularity assumption on the coefficients and data, then $D_{x'}u$ and $U$ are H\"{o}lder continuous.
\begin{theorem}\label{thm holder}
Let $\cQ$ be defined as above and the boundary condition on each subdomain $\cD_{j}$ be replaced with $C^{1,\mu}$. Let $\varepsilon\in (0,1)$ and $p\in(1,\infty)$. Assume that $A, B, g\in C^{\delta/2,\delta}(\overline{\cQ}_{j})$ with $\delta\in \Big(0,\mu/(1+\mu)\Big]$, and $f\in L_{\infty}(\cQ)$. If $u\in \mathcal{H}_{p}^{1}(\cQ)$ is a weak solution to \eqref{systems} in $\cQ$, then $u\in C^{1/2,1}(\overline{\cQ_{j}}\cap((-T+\varepsilon,0)\times\cD_{\varepsilon}))$, $j=1,\ldots,M$. Moreover, for any fixed $z_{0}\in (-T+\varepsilon,0)\times\cD_{\varepsilon}$, there exists a coordinate system associated with $z_{0}$ such that for all $z\in (-T+\varepsilon,0)\times\cD_{\varepsilon}$, we have
\begin{align}\label{Lip Du}
&|D_{x'}u(z_{0})-D_{x'}u(z)|+|U(z_{0})-U(z)|\nonumber\\
&\leq N|z_{0}-z|_{p}^{\delta'}
\bigg(\sum_{j=1}^{M}|g|_{\delta/2,\delta;\overline{\cQ}_{j}}
+\|f\|_{L_{\infty}(\cQ)}+\|u\|_{L_{p}(\cQ)}+\|Du\|_{L_{1}(\cQ)}\bigg),
\end{align}
and
\begin{align}\label{angle u t}
\langle u\rangle_{1+\delta';Q^-_{1/2}}\leq N\bigg(\sum_{j=1}^{M}|g|_{\delta/2,\delta;\overline{\cQ}_{j}}
+\|f\|_{L_{\infty}(\cQ)}+\|u\|_{L_{p}(\cQ)}+\|Du\|_{L_{1}(\cQ)}\bigg),
\end{align}
where $\delta'=\min\{\delta,2\gamma_{0}-1\}$, $N$ depends on $n$, $d$, $M$, $\delta$, $\mu$, $\nu$, $\Lambda$, $\varepsilon$, $p$, $\|A\|_{C^{\delta/2,\delta}(\overline{\cQ}_{j})}$, $\|B\|_{C^{\delta/2,\delta}(\overline{\cQ}_{j})}$, and the $C^{1,\mu}$ and $C^{\gamma_{0}}$ norms of $\cQ_{j}$ with respect to $x$ and $t$, respectively.
\end{theorem}

\section{Preliminaries}\label{preliminaries}
In this section, we first introduce the notation and definitions. Then we prove some properties of our domain, coefficients, and data. Finally we establish the existence and $L_p$-estimates of solutions to parabolic systems with coefficients satisfying certain regularity assumptions. Besides, we also prove some auxiliary estimates that will be used in the proofs of our main results.

\subsection{Notation and definitions}

We follow most of the notation in \cite{dx} and \cite{d}. We write $z=(t,x)$, $z'=(t,x')$, and $x=(x^{1},\ldots,x^{d})=(x',x^{d})$, where $d\geq2$. We denote
$$B_{r}(x):=\{y\in\mathbb R^{d}: |y-x|<r\},\quad B'_{r}(x'):=\{y'\in\mathbb R^{d-1}: |y'-x'|<r\},$$
$$Q_{r}^{-}(t,x):=(t-r^2,t)\times B_{r}(x),\quad Q_{r}^{-'}(t,x):=(t-r^2,t)\times B'_{r}(x'),$$
$$
Q^{+}_{r}(t,x):=(t,t+r^2)\times B_{r}(x),\quad Q_{r}(t,x):=(t-r^2,t+r^2)\times B_{r}(x),
$$
and the parabolic distance between two points $z_1=(t_1,x_1)$ and $z_2=(t_2,x_2)$ by
\begin{equation}\label{para dis}
|z_1-z_2|_{p}:=\max\left\{|t_1-t_2|^{1/2},|x_1-x_2|\right\}.
\end{equation}
We use $B_{r}:=B_{r}(0)$, $B'_{r}:=B'_{r}(0')$, $Q_{r}^{-}:=Q_{r}^{-}(0,0)$, $Q^{-'}_{r}:=Q^{-'}_{r}(0,0)$,  $\cQ_{r}^{-}(t,x):=\cQ\cap Q_{r}^{-}(t,x)$, and $\cQ_{r}(t,x):=\cQ\cap Q_{r}(t,x)$ for abbreviation, respectively. The parabolic boundary of $\cQ=(a,b)\times\cD$ is defined by
$$\partial_{p}\cQ=((a,b)\times\partial\cD)\cup(\{a\}\times\overline \cD).$$
The following notation will also be used:
$$D_{t}u=u_{t},\quad D_{x'}u=u_{x'},\quad DD_{x'}u=u_{xx'}.$$
For a function $f$ defined in $\mathbb R^{d+1}$, we set
$$(f)_{\cQ}=\frac{1}{|\cQ|}\int_{\cQ}f(t,x)\ dx\ dt=\fint_{\cQ}f(t,x)\ dx\ dt,$$
where $|\cQ|$ is the $d+1$-dimensional Lebesgue measure of $\cQ$. For $\gamma\in(0,1]$, we denote the $C^{\gamma/2,\gamma}$ semi-norm by
$$[u]_{\gamma/2,\gamma;\cQ}:=\sup_{\begin{subarray}{1}(t,x),(s,y)\in\cQ\\
(t,x)\neq (s,y)
\end{subarray}}\frac{|u(t,x)-u(s,y)|}{|t-s|^{\gamma/2}+|x-y|^{\gamma}},$$
and the $C^{\gamma/2,\gamma}$ norm by
$$|u|_{\gamma/2,\gamma;\cQ}:=[u]_{\gamma/2,\gamma;\cQ}+|u|_{0;\cQ},\quad \text{where}\,\,|u|_{0;\cQ}=\sup_{\cQ}|u|.$$
We define
$$[u]_{(1+\gamma)/2,1+\gamma;\cQ}:=[Du]_{\gamma/2,\gamma;\cQ}+\langle u\rangle_{1+\gamma;\cQ}$$
and
$$|u|_{(1+\gamma)/2,1+\gamma;\cQ}:=[u]_{(1+\gamma)/2,1+\gamma;\cQ}+|u|_{0;\cQ},$$
where
$$\langle u\rangle_{1+\gamma;\cQ}:=\sup_{\begin{subarray}{1}(t,x),(s,x)\in\cQ\\
\quad t\neq s
\end{subarray}}\frac{|u(t,x)-u(s,x)|}{|t-s|^{(1+\gamma)/2}}.$$
Next we define the semi-norm
$$[u]_{z',(1+\gamma)/2,1+\gamma;\cQ}:=[D_{x'}u]_{z',\gamma/2,\gamma;\cQ}+\langle u\rangle_{1+\gamma;\cQ}$$
and the norm
$$|u|_{z',(1+\gamma)/2,1+\gamma;\cQ}:=[u]_{z',(1+\gamma)/2,1+\gamma;\cQ}+|u|_{0;\cQ}+|Du|_{0;\cQ},$$
where
$$[D_{x'}u]_{z',\gamma/2,\gamma;\cQ}:=\sup_{\begin{subarray}{1}(t,x),(s,y)\in\cQ\\
x^{d}=y^{d}, (t,x)\neq(s,y)
\end{subarray}}\frac{|D_{x'}u(t,x)-D_{x'}u(s,y)|}{|t-s|^{\gamma/2}+|x-y|^{\gamma}}.$$
Denote $C_{z'}^{(1+\gamma)/2,1+\gamma}$ by the set of all bounded measurable functions $u$ for which $Du$ are bounded and continuous in $\cQ$ and $[u]_{z',(1+\gamma)/2,1+\gamma;\cQ}<\infty$.

We introduce some Lebesgue spaces which will be utilized throughout the paper. For $p\in(1,\infty)$, we denote
$$W_{p}^{1,2}(\cQ):=\{u: u, u_t, Du, D^2u\in L_{p}(\cQ)\}.$$
We define the solution spaces $\mathcal{H}_{p}^{1}(\cQ)$ as follows. Set
$$\mathbb{H}_{p}^{-1}(\cQ):=\bigg\{f: f=\sum_{|\alpha|\leq1}D^{\alpha}f_{\alpha}, f_{\alpha}\in L_{p}(\cQ)\bigg\},$$
$$\|f\|_{\mathbb{H}_{p}^{-1}(\cQ)}:=\inf\bigg\{\sum_{|\alpha|\leq1}\|f_{\alpha}\|_{L_{p}(\cQ)}: f=\sum_{|\alpha|\leq1}D^{\alpha}f_{\alpha}\bigg\},$$
and
$$\mathcal{H}_{p}^{1}(\cQ):=\{u: u_{t}\in \mathbb{H}_{p}^{-1}(\cQ), D^{\alpha}u\in L_{p}(\cQ), 0\leq|\alpha|\leq1\},$$
$$\|u\|_{\mathcal{H}_{p}^{1}(\cQ)}:=\|u_{t}\|_{\mathbb{H}_{p}^{-1}(\cQ)}+\sum_{|\alpha|\leq1}\|D^{\alpha}u\|_{L_{p}(\cQ)}.$$
Define $C_{0}^{\infty}([-1,0]\times\cD)$ to be the collection of infinitely differentiable functions $\phi:=\phi(t,x)$ with compact supports in $[-1,0]\times\cD$. Finally, we set $\mathcal{\mathring{H}}_{p}^{1}((-1,0)\times\cD)$ to be the closure of $C_{0}^{\infty}([-1,0]\times\cD)$ in $\mathcal{H}_{p}^{1}((-1,0)\times\cD)$.

\begin{definition}\label{piece Dini}
We say that a continuous increasing function $\omega: [0,1]\rightarrow\mathbb R$ satisfies the Dini condition provided that $\omega(0)=0$ and
$$\int_{0}^{r}\frac{\omega(s)}{s}\ ds<+\infty,\quad\forall~r\in(0,1).$$
\end{definition}

\begin{definition}\label{def Dini}
Let $\cD\subset\mathbb R^{d}$ be open and bounded. We say that $\partial\cD$ is $C^{1,\text{Dini}}$ if for each point $x_{0}\in\partial\cD$, there exists $R_0\in (0,1/8)$ independent of $x_{0}$ and a $C^{1,\text{Dini}}$ function (i.e., $C^{1}$ function whose first derivatives are Dini continuous) $\varphi: B'_{R_0}\rightarrow\mathbb R$ such that (upon relabeling and reorienting the coordinates if necessary) in a new coordinate system $(x',x^{d})$, $x_{0}$ becomes the origin,
$$
\cD_{R_{0}}(0)=\{x\in B_{R_{0}}: x^{d}>\varphi(x')\},
\quad\varphi(0')=0,\quad\nabla_{x'}\varphi(0')=0,
$$
and $\nabla_{x'}\varphi$ has a modulus of continuity $\omega_0$, which is increasing, concave, independent of $x_0$, and satisfies the Dini condition.
\end{definition}

\subsection{Some auxiliary estimates}\label{subsection domain}
Under the same setting as in Section \ref{prob formulation}, we have the following result.
\begin{lemma}\label{volume}
There exists a constant $N$, depending on $d, l$, the $C^{1,\text{Dini}}$ characteristics of $h_{j}(t,\cdot)$ for fixed $t\in(-1,0)$, and $C^{\gamma_0}$-norms of $h_{j}(\cdot,x')$ for fixed $x'\in B'_{1}(0')$, $1\leq j\leq l+1$, such that
$$|(\cQ_{j}\Delta\Omega_{j})\cap Q_{r}^{-}(z_{0})|\leq N\omega_1(r)r^{d+2}+Nr^{d+1+2\gamma_{0}}\ \ \text{when}\ 0<r<r_0:=\frac 2 3\int_0^{R_0/2}\omega_0'(s)s\,ds,$$
where $\cQ_{j}\Delta\Omega_{j}=(\cQ_{j}\setminus\Omega_{j})\cup(\Omega_{j}\setminus \cQ_{j})$, $\omega_0'$ denotes the left derivative of $\omega_0$, and $\omega_1=\omega_1(r)$ is a Dini function derived from $\omega_0$ in Definition \ref{def Dini}.
\end{lemma}

\begin{proof}
Let $(t_0,x',h_j(t_0,x'))\in Q_r^{-}(t_{0},x_0)$ for some $x'\in B'_{1}$. We first prove that $|\nabla_{x'}h_{j}(t_0,x')|$ in $B'_{r}(x_0')$ is bounded by $\omega_1(r):=2\omega_0(2r+R)$, where $R$ is a fixed number only depending on $r$. This is based on the arguments in the proof of \cite[Lemma 2.3]{dx}. Indeed, we denote the supremum of $|\nabla_{x'}h_{j}(t_0,x')|$ in $B'_{r}(x_0')$ by $\mathbf{S}$. Then for fixed $(t_0,y',h_{j}(t_0,y'))\in Q_{r}^{-}(t_0,x_0)$, we obtain from \eqref{eq10.42} and $\nabla_{x'}h_{j_{0}}(t_0,x'_0)=0$ that
$$
|h_{j_{0}}(t_0,y')-h_{j_{0}}(t_0,x_0')|\le r,\quad
|\nabla_{x'}h_{j_{0}}(t_0,y')|\le \omega_{0}(2r),\quad
|\nabla_{x'}h_{j}(t_0,y')|\ge \mathbf{S}-\omega_{0}(2r).
$$
Then in view of \eqref{eq10.42}, we have for any $R\in(0,1/8)$,
\begin{equation}\label{inetegral00}
\int_{0}^{R}(\mathbf{S}-2\omega_0(2r+s))\ ds\leq 3r.
\end{equation}
The maximum of the left-hand side of \eqref{inetegral00} with respect to $R$ is attained when $2\omega_0(2r+R)=\mathbf{S}$. This yields
\begin{equation*}
R\omega_0(2r+R)-\int_{0}^{R}\omega_0(2r+s)\ ds\leq 3r/2.
\end{equation*}
So we have
\begin{equation}\label{int omega}
\int_{0}^{R}\omega'_0(2r+s)s\ ds\leq 3r/2.
\end{equation}
In order to obtain an upper bound for $\mathbf{S}$, we use \eqref{int omega} to fix the number $R=R(r)(>2r)$ such that
\begin{equation*}
\int_{0}^{R}\omega_0'(2r+s)s\ ds=3r/2.
\end{equation*}
We henceforth get
$$\mathbf{S}=2\omega_0(2r+R)=:\omega_1(r),$$
which is a Dini function on $(0,r_0)$. See the proof of \cite[Lemma 2.3]{dx}.

Now, for any $t\in(t_{0}-r^{2},t_{0})$, by using the triangle inequality and $h_{j}(\cdot,x')\in C^{\gamma_{0}}$, we have
\begin{align*}
|h_{j}(t,x')-h_{j}(t_{0},x'_{0})|\leq|h_{j}(t,x')-h_{j}(t_{0},x')|+|h_{j}(t_{0},x')-h_{j}(t_{0},x'_{0})|\leq Nr^{2\gamma_{0}}+\omega_{1}(r)r.
\end{align*}
We thus obtain
\begin{align*}
|(\cQ_{j}\Delta\Omega_{j})\cap Q_{r}^{-}(t_{0},x_{0})|\leq N(r^{2\gamma_{0}}+\omega_{1}(r)r)r^{d+1}\leq Nr^{d+1+2\gamma_{0}}+N\omega_1(r)r^{d+2}.
\end{align*}
The lemma is proved.
\end{proof}

Let $\hat{A}^{(j)}\in\mathcal{A}$ be a constant function in $\cQ_{j}$ which corresponds to the definition of $\omega_{A}(r)$ in \eqref{def omega A}. Similarly, $\hat{B}^{(j)}$ and $\hat{g}^{(j)}$ are defined in $\cQ_{j}$. We define the piecewise constant (matrix-valued) functions
\begin{align*}
\bar{A}(t,x)=
\hat{A}^{(j)},\quad (t,x)\in\Omega_{j}.
\end{align*}
We remark that $\bar{A}(t,x)$ only depends on $x^{d}$. Using $\hat{B}^{(j)}$ and $\hat{g}^{(j)}$, we similarly define piecewise constant functions $\bar{B}$ and $\bar{g}$. From Lemma \ref{volume} and the boundedness of $A$, we have
\begin{align}\label{est A}
\fint_{Q_{r}^{-}(z_{0})}|\hat{A}-\bar{A}|\ dz
&\leq Nr^{-d-2}\sum_{j=1}^{l+1}|(\cQ_{j}\Delta\Omega_{j})\cap Q_{r}^{-}(z_{0})|\nonumber\\
&\leq N\omega_1(r)+Nr^{2\gamma_{0}-1}=:N\hat\omega_{1}(r),
\end{align}
where $\hat\omega_{1}(r):=\omega_1(r)+r^{2\gamma_{0}-1}$ is a Dini function. This is also true for $\hat{B}$ and $\hat{g}$.

We now turn to the $\mathcal{H}_{p}^{1}$-estimate for parabolic equations with variably partially small BMO (bounded mean oscillation) coefficients (see \cite{dk2}): there exists a sufficiently small constant $\gamma_{0}=\gamma_{0}(d,n,p,\nu)\in (0,1/2)$ and a constant $r_{0}\in(0,1)$ such that for any $r\in(0,r_{0})$ and $(t_{0},x_{0})\in Q_{1}^{-}$ with $B_{r}(x_{0})\subset B_{1}$, in a coordinate system depending on $(t_{0},x_{0})$ and $r$, one can find a $\bar{A}=\bar{A}(x^{d})$ satisfying
\begin{equation}\label{BMO}
\fint_{Q_{r}^{-}(t_{0},x_{0})}|A(t,x)-\bar A(x^{d})|\ dx\ dt\leq\gamma_{0}.
\end{equation}
We obtain the following lemma from \cite[Theorem 8.2]{dk2} by a similar localization argument that led to \cite[Lemma 4, Corollary 3]{d}, the interpolation inequality, and iteration arguments.

\begin{lemma}\label{lem loc lq}
Let $0<p<1<q<\infty$. Assume $A$ satisfies \eqref{BMO} with a sufficiently small constant $\gamma_{0}=\gamma_{0}(d,n,p,q,\nu,\Lambda)\in (0,1/2)$ and $u\in \mathcal{H}_{p,\text{loc}}^{1}$ satisfies \eqref{systems} in $Q_{1}^{-}$, where $f,g\in L_{q}(Q_{1}^{-})$. Then
$$\|u\|_{\mathcal{H}_{q}^{1}(Q_{1/2}^{-})}\leq N(\|u\|_{L_{p}(Q_{1}^{-})}+\|g\|_{L_{q}(Q_{1}^{-})}+\|f\|_{L_{q}(Q_{1}^{-})}).$$
In particular, if $q>d+2$, it holds that
$$|u|_{\gamma/2,\gamma;Q_{1/2}^{-}}\leq N(\|u\|_{L_{p}(Q_{1}^{-})}+\|g\|_{L_{q}(Q_{1}^{-})}+\|f\|_{L_{q}(Q_{1}^{-})}),$$
where $\gamma=1-(d+2)/q$ and $N$ depends on $n,d,\nu,\Lambda,p,q$, and $r_{0}$.
\end{lemma}

In the proofs below, we will also use the $\mathcal{H}_{p}^{1}$-solvability for parabolic systems with leading coefficients which satisfy \eqref{BMO} in $Q_{1}^{-}$. For this we choose a cut-off function $\eta\in C_{0}^{\infty}(B_{1})$ with
$$0\leq\eta\leq1,\quad \eta\equiv1~\mbox{in}~B_{3/4},\quad|\nabla\eta|\leq 8.$$
Let $\tilde{\mathcal{P}\ }$ be the parabolic operator defined by
$$
\tilde{\mathcal{P}\ }u:=-u_{t}+D_{\alpha}(\tilde{A}^{\alpha\beta}D_{\beta}u),
$$
where $\tilde{A}^{\alpha\beta}=\eta A^{\alpha\beta}(t,x)+\nu(1-\eta)\delta_{\alpha\beta}\delta_{ij}$, $\delta_{\alpha\beta}$ and $\delta_{ij}$ are the Kronecker delta symbols. Then for sufficiently small $\gamma$, the coefficients $\tilde{A}^{\alpha\beta}(t,x)$ and the boundary $\partial B_{1}$ satisfy the Assumption 8.1 ($\gamma$) in \cite{dk2}. By \cite[Theorem 8.2]{dk2} (or Lemma \ref{sol weight} below), we have
\begin{lemma}\label{solvability}
For any $p\in(1,\infty)$, $g, f\in L_{p}(Q_{1}^{-})$, the following hold.
\begin{enumerate}
\item
For any $u\in \mathcal{\mathring{H}}_{p}^{1}(Q_{1}^{-})$ satisfying
\begin{align}\label{approxi sol}
\tilde{\mathcal{P}\ }u=\Div g+f&\quad\mbox{in}~Q_{1}^{-},
\end{align}
we have
\begin{align}\label{est H u}
\|u\|_{\mathcal{H}_{p}^{1}(Q_{1}^{-})}\leq N\big(\|g\|_{L_{p}(Q_{1}^{-})}+\|f\|_{L_{p}(Q_{1}^{-})}\big),
\end{align}
where $N$ depends on $d,n,p,\nu,\Lambda$, and $r_{0}$.

\item
For any $g, f\in L_{p}(Q_{1}^{-})$, there exists a unique solution $u\in \mathcal{\mathring{H}}_{p}^{1}(Q_{1}^{-})$ of \eqref{approxi sol} with the initial data $u(-1,\cdot)\equiv0$ in $B_{1}$. Furthermore, $u$ satisfies \eqref{est H u}.
\end{enumerate}
\end{lemma}

In addition to the above estimates, we will also need to consider systems with coefficients depending only on $x^{d}$. Denote
$$\mathcal{P}_{0}u:=-u_{t}+D_{\alpha}(\bar{A}^{\alpha\beta}(x^{d})D_{\beta}u),$$
and
$$\bar{U}:=\bar{A}^{d\beta}(x^{d})D_{\beta}u.$$

\begin{lemma}\label{lemma xn}
Let $p\in(0,\infty)$. Assume $u\in C_{\text{loc}}^{0,1}$ satisfies $\mathcal{P}_{0}u=0$ in $Q_{1}^{-}$. Then there exists a constant $N=N(n,d,p,\nu,\Lambda)$ such that
\begin{align}\label{DDu}
[u]_{C^{1/2,1}(Q_{1/2}^{-})}\leq N\|u\|_{L_{p}(Q_{1}^{-})}, \quad[D_{x'}u]_{C^{1/2,1}(Q_{1/2}^{-})}\leq N\|D_{x'}u\|_{L_{p}(Q_{1}^{-})},
\end{align}
and
\begin{align}\label{DU}
[\bar{U}]_{C^{1/2,1}}(Q_{1/2}^{-})\leq N\|Du\|_{L_{p}(Q_{1}^{-})}.
\end{align}
\end{lemma}

\begin{proof}
By using Lemma \ref{lem loc lq}, the Sobolev embedding theorem, the interpolation inequality, and iteration arguments, we have
\begin{align}\label{u infinite}
\|u\|_{L_{\infty}(Q_{4/5}^{-})}\leq N\|u\|_{L_{p}(Q_{1}^{-})},\quad p>0.
\end{align}
For fixed $t\in(-1,0)$, we define the finite difference quotient
$$
\delta_{h,k}f(t,x):=\frac{f(t,x+he_{k})-f(t,x)}{h},
$$
where $k=1,\dots,d-1$, $0<|h|<1/12$. Since $\bar{A}^{\alpha\beta}(x^{d})$ are independent of $x'$, we have $\mathcal{P}_{0}(\delta_{h,k}u)=0$ in $Q_{1}^{-}$. Then in view of Lemma \ref{lem loc lq} and \eqref{u infinite}, we obtain
\begin{align*}
\|\delta_{h,k}u\|_{\mathcal{H}_{q}^{1}(Q_{1/2}^{-})}\leq N\|\delta_{h,k}u\|_{L_{2}(Q_{2/3}^{-})}\leq N\|D_{x'}u\|_{L_{2}(Q_{3/4}^{-})},\quad\forall~ q>1.
\end{align*}
Letting $h\rightarrow0$, we obtain
\begin{align}\label{Dx'u WW}
\|D_{x'}u\|_{\mathcal{H}_{q}^{1}(Q_{1/2}^{-})}\leq N\|D_{x'}u\|_{L_{2}(Q_{3/4}^{-})},\quad \forall\ q>1.
\end{align}
On the other hand, from \cite[Lemma 3.3]{dk3}, we have
\begin{equation}\label{es ut}
\|u_{t}\|_{L_{2}(Q_{2/3}^{-})}\leq N\|Du\|_{L_{2}(Q_{3/4}^{-})}.
\end{equation}
Observing that $\mathcal{P}_{0}(u_{t})=0$ in $Q_{1}^{-}$, and using Lemma \ref{lem loc lq},  \eqref{u infinite}, and \eqref{es ut}, we get
\begin{equation}\label{ut Hq1}
\|u_{t}\|_{\mathcal{H}_{q}^{1}(Q_{1/2}^{-})}\leq N\|u_{t}\|_{L_{2}(Q_{2/3}^{-})}\leq N\|Du\|_{L_{2}(Q_{3/4}^{-})},\quad\forall~ q>1.
\end{equation}
Hence, by the Sobolev embedding theorem for $q>d+2$, we have
\begin{equation}\label{infty ut}
\|u_{t}\|_{L_{\infty}(Q_{1/2}^{-})}\leq N\|Du\|_{L_{2}(Q_{3/4}^{-})}.
\end{equation}

Now notice that in $Q_{1}^{-}$,
\begin{align}\label{Dx U}
D_{d}\bar{U}=u_{t}-\sum_{\alpha=1}^{d-1}\sum_{\beta=1}^{d}\bar{A}^{\alpha\beta}D_{\alpha\beta}u,
\  D_{x'}\bar{U}=\sum_{\beta=1}^{d}\bar{A}^{d\beta}D_{x'}D_{\beta}u,
\  \bar{U}_t=\sum_{\beta=1}^{d}\bar{A}^{d\beta}D_t D_{\beta}u.
\end{align}
Therefore, it follows from \eqref{Dx'u WW}, \eqref{ut Hq1}, and \eqref{Dx U} that
$$\|D_{x'}u\|_{\mathcal{H}_{q}^{1}(Q_{1/2}^{-})}+\|\bar{U}\|_{\mathcal{H}_{q}^{1}(Q_{1/2}^{-})}\leq N\|Du\|_{L_{2}(Q_{3/4}^{-})}.$$
Then by the Sobolev embedding theorem for $q>d+2$, $\bar{A}^{dd}(x^{d})\geq\nu$, and the definition of $\bar{U}$, we have
\begin{equation}\label{es Du infty}
\|Du\|_{L_{\infty}(Q_{1/2}^{-})}\leq N\|Du\|_{L_{2}(Q_{3/4}^{-})}.
\end{equation}
By using \eqref{infty ut}, \eqref{es Du infty}, Lemma \ref{lem loc lq},  and \eqref{u infinite}, we have
\begin{equation}\label{est ut Du}
\|u_{t}\|_{L_{\infty}(Q_{1/2}^{-})}+\|Du\|_{L_{\infty}(Q_{1/2}^{-})}\leq N\|Du\|_{L_{2}(Q_{3/4}^{-})}\leq C\|u\|_{L_{2}(Q_{4/5}^{-})}\leq C\|u\|_{L_{p}(Q_{1}^{-})},\quad\forall~p>0.
\end{equation}
Recalling that the coefficients of $\mathcal{P}_{0}$ are independent of $x'$, we henceforth have
$$
\mathcal{P}_{0}(D_{x'}u)=0\quad \quad\text{in}\  Q_{1}^{-}.
$$
Replacing $u$ with $D_{x'}u$ in \eqref{est ut Du}, we get
\begin{equation}\label{est Dut}
\|D_{x'}u_{t}\|_{L_{\infty}(Q_{1/2}^{-})}+\|DD_{x'}u\|_{L_{\infty}(Q_{1/2}^{-})}\leq N\|D_{x'}u\|_{L_{p}(Q_{1}^{-})},\quad \forall~p>0.
\end{equation}
We thus obtain \eqref{DDu} by using \eqref{est ut Du} and \eqref{est Dut}.

Next we prove \eqref{DU}. By using \eqref{es Du infty}, the interpolation inequality,  and iteration arguments, we obtain
\begin{equation}\label{est Du p}
\|Du\|_{L_{\infty}(Q_{1/2}^{-})}\leq N\|Du\|_{L_{p}(Q_{1}^{-})},\quad \forall~p>0.
\end{equation}
Now by using the fact that the coefficients of $\mathcal{P}_{0}$ are independent of $t$, we have
$$
\mathcal{P}_{0}(u_{t})=0\quad\text{in}\  Q_{1}^{-}.
$$
Replacing $u$ with $u_{t}$ in \eqref{est ut Du} with a slightly smaller domain, using \eqref{infty ut} and \eqref{est Du p}, we have
\begin{equation}\label{es Dut infty}
\|Du_{t}\|_{L_{\infty}(Q_{1/2}^{-})}\leq N\|Du\|_{L_{p}(Q_{1}^{-})},\quad \forall~p>0.
\end{equation}
Therefore, \eqref{DU} is a consequence of \eqref{DDu}, \eqref{Dx U}, and \eqref{es Dut infty}. The lemma is proved.
\end{proof}

We remark that the same proofs of Lemmas \ref{lem loc lq}--\ref{lemma xn} give similar results for the adjoint operator of $\mathcal{P}$. We close this section by giving the following two lemmas.

\begin{lemma}{\cite[Lemma 2.7]{dk}}\label{lemma omiga}
Let $\omega$ be a nonnegative bounded function. Suppose there is $c_{1},c_{2}>0$ and $0<\kappa<1$ such that for $\kappa t\leq s\leq t$ and $0<t<r$,
\begin{align}\label{equivalence}
c_{1}\omega(t)\leq \omega(s)\leq c_{2}\omega(t).
\end{align}
Then, we have
$$\sum_{i=0}^{\infty}\omega(\kappa^{i}r)\leq N\int_{0}^{r}\frac{\omega(t)}{t}\ dt,$$
where $N=N(\kappa,c_{1},c_{2})$.
\end{lemma}

\begin{lemma}\label{lemma weak}
Let $\cQ=(-T,0)\times\cD$ be a bounded domain in $\mathbb R^{d+1}$. For fixed $t$ and all $x\in\overline{\cD}$,
\begin{equation*}
|\cD\cap B_{r}(x)|\geq A_{0}r^{d}\quad\mbox{for}~r\in(0,\mbox{diam}\ \cD],
\end{equation*}
where $A_{0}>0$ is a constant. Let $p\in (1,\infty)$ and $S$ be a bounded linear operator on $L_{p}(\cQ)$. Suppose that for any $\bar{z}\in \cQ$ and $0<r<\mu~ \mbox{diam}\ \cD$,  we have
$$
\int_{\cQ\setminus Q_{cr}(\bar{z})}|Sb|\leq C_{0}\int_{\cQ_{r}^{-}(\bar{z})}|b|
$$
whenever $b\in L_{p}(\cQ)$ is supported in $\cQ_{r}^{-}(\bar{z})$, $\int_{\cQ}b=0$, and $c>1$, $C_{0}>0$, $\mu\in(0,1)$ are constants. Then for $g\in L_{p}(\cQ)$ and any $s>0$, we have
$$|\{(t,x)\in \cQ: |Sg(t,x)|>s\}|\leq\frac{N}{s}\int_{\cQ}|g|,$$
where $N=N(d,c,C_{0},\cD,A_{0},\mu,\|S\|_{L_p\rightarrow L_p})$ is a constant.
\end{lemma}

Lemma \ref{lemma weak} is similar to \cite[Lemma 4.1]{dek}, where the proof is based on the Calder\'{o}n-Zygmund decomposition. Here we can modify the proof there by using the ``dyadic parabolic cube'' decomposition of $\cQ$. See also \cite[Theorem 11]{Ch90}.

\section{Proofs of Theorem \ref{thm1} and Corollary \ref{coro loc}}\label{proof thm1}
In this section, we give the proofs of Theorem \ref{thm1} and Corollary \ref{coro loc}. First of all, by using Campanato's characterization of H\"{o}lder continuous functions, the global $W_{p}^{1,2}$ estimate for the heat equation, the Sobolev embedding theorem, and Lemma \ref{lem loc lq}, we reduce the proof of Theorem \ref{thm1} to that of Proposition \ref{main prop}, which is about the parabolic systems without lower-order terms . Then we prove some auxiliary estimates that play key roles in deriving  an a priori estimate of the modulus of continuity of $(D_{x'}u,U)$. With the above preparations, we complete the proof of Proposition \ref{main prop} by discussing two cases since our argument and estimates depend on the coordinate system. Finally, we prove Corollary \ref{coro loc} using a duality argument and Theorem \ref{thm1}.

\subsection{Simplified problem}\label{subsec prb}
We first reduce the estimate of $[u]_{1/2,1}$ to the estimate of $\|Du\|_{L_{\infty}}$ by using the following lemma.
\begin{lemma}
Let $u$ be a weak solution to \eqref{systems} in $Q_{1}^{-}$. Suppose that $\|u\|_{L_{\infty}(Q_{1/2}^{-})}<\infty$ and $\|Du\|_{L_{\infty}(Q_{1/2}^{-})}<\infty$. Then
\begin{equation*}
[u]_{1/2,1;Q_{1/4}^{-}}\leq N\left(\|u\|_{L_{\infty}(Q_{1/2}^{-})}+\|Du\|_{L_{\infty}(Q_{1/2}^{-})}
+\|f\|_{L_{\infty}(Q_{1/2}^{-})}+\|g\|_{L_{\infty}(Q_{1/2}^{-})}\right).
\end{equation*}
\end{lemma}

\begin{proof}
The lemma follows from a similar argument that led to \cite[Lemma 6]{d} by using Campanato's characterization of H\"{o}lder continuous functions (see \cite[Lemma 4.3]{lgm}) and a variant of the parabolic Poincar\'{e} inequality (see \cite[Lemma 3]{d}).
\end{proof}

Next we show that it suffices to consider the parabolic systems without lower-order terms. Rewrite \eqref{systems} as
$$-u_{t}+D_{\alpha}(A^{\alpha\beta}D_{\beta}u)=\Div (g-Bu)+f-\hat{B}^{\alpha}D_{\alpha}u-Cu\quad \mbox{in}~\cQ=Q_{1}^{-}.$$
Let $v\in W_{p}^{1,2}(Q_{1}^{-})$ satisfy
\begin{align*}
\begin{cases}
-v_{t}+\Delta v=(f-\hat{B}^{\alpha}D_{\alpha}u-Cu)\chi_{Q_{1/2}^{-}}
&\ \mbox{in}~Q_{1}^{-},\\
v=0&\ \mbox{on}~\partial_p Q_1^-.
\end{cases}
\end{align*}
Then by the global $W_{p}^{1,2}$ estimate for the heat equation, we have
\begin{align}\label{W2p v}
\|v\|_{W_{p}^{1,2}(Q_{1}^{-})}\leq N\left(\|u\|_{L_{p}(Q_{1/2}^{-})}
+\|Du\|_{L_{p}(Q_{1/2}^{-})}
+\|f\|_{L_{\infty}(Q_{1/2}^{-})}\right).
\end{align}
By Lemma \ref{lem loc lq}, we have for some $q>d+2$,
\begin{equation}\label{est u H1q}
\|u\|_{\mathcal{H}_{q}^{1}(Q_{1/2}^{-})}\leq N\left(\|u\|_{L_{p}(Q_{1}^{-})}
+\|g\|_{L_{\infty}(Q_{1}^{-})}+\|f\|_{L_{\infty}(Q_{1}^{-})}\right).
\end{equation}
By the Sobolev embedding theorem for $q>d+2$, we obtain $u\in C^{\beta/2,\beta}(Q_{1/2}^{-})$ and
$$
\|u\|_{C^{\beta/2,\beta}(Q_{1/2}^{-})}\leq N\left(\|u\|_{L_{p}(Q_{1}^{-})}
+\|g\|_{L_{\infty}(Q_{1}^{-})}+\|f\|_{L_{\infty}(Q_{1}^{-})}\right),
$$
where $\beta=1-(d+2)/{q}$. Now coming back to \eqref{W2p v}, replacing $p$ with $q$, and using \eqref{est u H1q}, we get $v\in C^{(1+\beta)/2,1+\beta}(Q_{1/2}^{-})$ with
\begin{equation}\label{est Dv Holder}
\|v\|_{C^{(1+\beta)/2,1+\beta}(Q_{1/2}^{-})}\leq N\left(\|u\|_{L_{p}(Q_{1}^{-})}+\|g\|_{L_{\infty}(Q_{1}^{-})}+\|f\|_{L_{\infty}(Q_{1}^{-})}\right).
\end{equation}
Denote $g':=g-Bu+(I-A)Dv$ and $w:=u-v$, then $w$ satisfies
$$-w_{t}+D_{\alpha}(A^{\alpha\beta}D_{\beta}w)=\Div g'\quad \mbox{in}~Q_{1/2}^{-},$$
where
$$\|g'\|_{L_{\infty}(Q_{1/2}^{-})}\leq N\left(\|u\|_{L_{p}(Q_{1}^{-})}+\|g\|_{L_{\infty}(Q_{1}^{-})}+\|f\|_{L_{\infty}(Q_{1}^{-})}\right).$$
Moreover, $g'$ is of piecewise Dini mean oscillation satisfying
\begin{align*}
&\omega_{g'}(r)\\
&\leq N(\Lambda)\big(\omega_{g}(r)+\omega_{B}(r)\|u\|_{L_\infty(Q_{1/2}^{-})}+r^{\beta}
[u]_{\beta/2,\beta;Q_{1/2}^{-}}+\omega_{A}(r)\|Dv\|_{L_\infty(Q_{1/2}^{-})}+r^{\beta}[Dv]_{\beta/2,\beta;Q_{1/2}^{-}}\big)\\
&\leq N\Big(\omega_{g}(r)+\omega_{B}(r)\|u\|_{L_\infty(Q_{1/2}^{-})}
+(\omega_{A}(r)+r^{\beta})\cdot \big(\|u\|_{L_{p}(Q_{1}^{-})}+\|g\|_{L_{\infty}(Q_{1}^{-})}
+\|f\|_{L_{\infty}(Q_{1}^{-})}\big)\Big).
\end{align*}
Therefore, bearing in mind that $u=w+v$ and $v$ satisfies \eqref{est Dv Holder}, the results for $w$ yield these  for $u$.

Finally, we conclude that to finish the proof of Theorem \ref{thm1}, we only need to prove  the following proposition.

\begin{prop}\label{main prop}
Let $\varepsilon\in(0,1)$ and $p\in(1,\infty)$. Suppose that $A$ and $g$ are of piecewise Dini mean oscillation in $\cQ$, and $g\in L_
{\infty}(\cQ)$. If $u\in \mathcal{H}_{p}^{1}(\cQ)$ is a weak solution to
$$-u_{t}+D_{\alpha}(A^{\alpha\beta}\cD_{\beta}u)=\Div g\quad\mbox{in}~\cQ,$$
then $u\in C^{1/2,1}(\overline{{\cQ}_{j}}\cap((-T+\varepsilon,0)\times\cD_{\varepsilon}))$, $j=1,\ldots,M$, and for any fixed $t\in(-T+\varepsilon,0)$, $u(t,\cdot)$ is Lipschitz in $\cD_{\varepsilon}$.
\end{prop}

We will establish an a priori estimate of the modulus of continuity of $(D_{x'}u,U)$ by assuming that $u\in C^{0,1}(Q_{3/4}^{-})$, i.e., for each $t\in(-9/16,0)$, $u(t,\cdot)\in C^{0,1}(B_{3/4})$. The proof of Proposition \ref{main prop} is mainly based on Campanato's approach \cite{c,g}. The general case follows from an approximation argument and the technique of locally flattening the boundaries \cite[p. 2466]{dx}.

Fix $z_{0}=(t_{0},x_{0})\in \Big((-9/16,0)\times B_{3/4}\Big)\cap \cQ_{j_{0}}$, $0<r\leq 1/4$, and take a coordinate system associated with $(t_{0},x_{0})$ as in Subsection \ref{subsection domain}. Denote
\begin{align*}
\bar{\ \mathcal{P}_{z'_{0}}}u:=-u_{t}+D_{\alpha}(\bar{A}^{\alpha\beta}(z'_{0},x^{d})D_{\beta}u),
\end{align*}
where $z'_{0}=(t_{0},x'_{0})$. Next we prove several auxiliary lemmas which play important roles in the proof of Proposition \ref{main prop}.

\subsection{Auxiliary lemmas}
We will begin with a weak type-$(1,1)$ estimate. Before that, we need to modify the coefficients $\bar{A}^{\alpha\beta}(z'_{0},x^{d})$ to get the following parabolic operator defined by
$$
\tilde{\mathcal{P}\ }u:=-u_{t}+D_{\alpha}(\tilde{A}^{\alpha\beta}D_{\beta}u),
$$
where $\tilde{A}^{\alpha\beta}=\eta \bar{A}^{\alpha\beta}(z'_{0},x^{d})+\nu(1-\eta)\delta_{\alpha\beta}\delta_{ij}$ with $\eta\in C_{0}^{\infty}(B_{r}(x_{0}))$ satisfying
$$0\leq\eta\leq1,\quad\eta\equiv1~\mbox{in}~B_{2r/3}(x_{0}),
\quad|\nabla\eta|\leq {6}/{r}.$$
Then we can apply Lemma \ref{solvability} with a scaling to the operator $\tilde{\mathcal{P}\ }$.

\begin{lemma}\label{weak est barv}
Let $p\in(1,\infty)$. Let $v\in \mathcal{H}_{p}^{1}(Q_{r}^{-}(z_{0}))$ be a weak solution to the problem
\begin{align*}
\begin{cases}
\tilde{\mathcal{P}\ }v=\Div(F\chi_{Q_{r/2}^{-}(z_{0})})&\ \mbox{in}~Q_{r}^{-}(z_{0}),\\
v=0&\ \mbox{on}~\partial_p Q_{r}^{-}(z_{0}),
\end{cases}
\end{align*}
where $F\in L_{p}(Q_{r/2}^{-}(z_{0}))$. Then for any $s>0$, we have
\begin{align*}
|\{z\in Q_{r/2}^{-}(z_{0}): |Dv(z)|>s\}|\leq\frac{N}{s}\|F\|_{L_{1}(Q_{r/2}^{-}(z_{0}))},
\end{align*}
where $N=N(n,d,p,\nu)$.
\end{lemma}

\begin{proof}
The proof is a modification of \cite[Lemma 3.2]{dx}. We set $z_{0}=0$, $r=1$, $\bar{A}^{\alpha\beta}(x^{d}):=\bar{A}^{\alpha\beta}(0',x^{d})$, and $\bar{\mathcal{P}\ }:=\bar{\ \mathcal{P}_{0'}}$ for simplicity. Suppose $E=(E^{\alpha\beta}(x^{d}))$ is a $d\times d$ matrix with
\begin{align*}
E^{\alpha\beta}(x^{d})&=\delta_{\alpha\beta}~\  \mbox{for}~\alpha, \beta\in\{1,\ldots,d-1\};\quad E^{\alpha d}(x^{d})=\bar{A}^{d\alpha}(x^{d})~\  \mbox{for}~\alpha\in\{1,\ldots,d\};\\
E^{d\beta}(x^{d})&=0~ \ \mbox{for}~\beta\in\{1,\ldots,d-1\}.
\end{align*}
For any $\hat F\in L_{p}(Q_{1/2}^{-})$, let $F=E\hat F$ and solve for $v$. It follows from Lemma \ref{solvability} that $S:\hat F\to Dv$ is a bounded linear operator on $L_{p}(Q_{1/2}^{-})$. So we only need to prove that $S$ satisfies the hypothesis of Lemma \ref{lemma weak}. Set $c=24$ and fix $\bar{z}=(\bar{t},\bar{y})\in Q_{1/2}^{-}$, $0<r<1/4$. Let $\hat{b}\in L_{p}(Q_{1}^{-})$ be supported in $Q_{r}^{-}(\bar{z})\cap Q_{1/2}^{-}$ with mean zero, $b=E\hat{b}$, and $v_{1}\in \mathcal{H}_{p}^{1}(Q_{1}^{-})$ be the unique weak solution of
\begin{align*}
\begin{cases}
\tilde{\mathcal{P}\ }v_{1}=\Div b&\ \mbox{in}~Q_{1}^{-},\\
v_{1}=0&\ \mbox{on}~\partial_p Q_{1}^{-}.
\end{cases}
\end{align*}
For any $R\geq cr$ such that $Q_{1/2}^{-}\setminus Q_{R}(\bar{z})\neq\emptyset$ and $h\in C_{0}^{\infty}((Q_{2R}(\bar{z})\setminus Q_{R}(\bar{z}))\cap Q_{1/2}^{-})$, let $v_{0}\in \mathcal{H}_{p'}^{1}(Q_{1}^{-})$ be a weak solution of
\begin{align*}
\begin{cases}
\tilde{\mathcal{P}^{*}\ }v_{0}=\Div h&\ \mbox{in}~Q_{1}^{-},\\
v_{0}=0&\ \mbox{on}~((-1,0]\times \partial B_{1})\cup (\{0\}\times \overline{B_1}),
\end{cases}
\end{align*}
where ${1}/{p}+{1}/{p'}=1$ and $\tilde{\mathcal{P}^{*}\ }$ is the adjoint operator of $\tilde{\mathcal{P}\ }$ defined by
$$\tilde{\mathcal{P}^{*}\ }u:=u_{t}+D_{\beta}(\tilde{A}^{\beta\alpha}D_{\alpha}u).$$
In view of the definition of weak solutions and the assumption of $\hat b$, we have
\begin{align*}
\int_{Q_{1/2}^{-}}Dv_{1}\cdot h&=\int_{Q_{1/2}^{-}}Dv_{0}\cdot b=\int_{Q_{r}^{-}(\bar{z})\cap Q_{1/2}^{-}}
\left(
          D_{x'}v_{0}, V_{0}
 \right)\cdot\hat{b}\\
&=\int_{Q_{r}^{-}(\bar{z})\cap Q_{1/2}^{-}}
\left(
          D_{x'}v_{0}-D_{x'}v_{0}(\bar{z}), V_{0}-V_{0}(\bar{z})
 \right)\cdot\hat{b},
\end{align*}
where
$V_{0}=\bar{A}^{d\beta}(x^{d})D_{\beta}v_{0}$. Hence, we have
\begin{align}\label{esti Dv h}
&\left|\int_{(Q_{2R}(\bar{z})\setminus Q_{R}(\bar{z}))\cap Q_{1/2}^{-}}Dv_{1}\cdot h\right|\nonumber\\
&\leq \|\hat{b}\|_{L_{1}(Q_{r}^{-}(\bar{z})\cap Q_{1/2}^{-})}\left|\left|\left(
          D_{x'}v_{0}-D_{x'}v_{0}(\bar{z}), V_{0}-V_{0}(\bar{z})
 \right)\right|\right|_{L_{\infty}(Q_{r}^{-}(\bar{z})\cap Q_{1/2}^{-})}.
\end{align}
Moreover, we find that $v_{0}\in \mathcal{H}_{p'}^{1}(Q_{1}^{-})$ satisfies
$$\bar{\mathcal{P}^{*}\ }v_{0}=0\quad\mbox{in}~Q_{R/12}^{-}(\bar{z}),$$
where we recalled that $\eta\equiv1$ in $B_{2/3}$ and $B_{R/12}(\bar{y})\subset B_{2/3}$. By applying a similar argument that led to \eqref{DDu} and \eqref{DU} to the adjoint operator, and using a suitable scaling, $r\leq R/24$, and the $\mathcal{H}^1_{p}$ estimate, we have
\begin{align}\label{est Dv0 V0}
&\|D_{x'}v_{0}-D_{x'}v_{0}(\bar{z})\|_{L_{\infty}(Q_{r}^{-}(\bar{z})\cap Q_{1/2}^{-})}
+\|V_{0}-V_{0}(\bar{z})\|_{L_{\infty}(Q_{r}^{-}(\bar{z})\cap Q_{1/2}^{-})}\nonumber\\
&\leq Nr([D_{x'}v_{0}]_{C^{1/2,1}(Q_{R/24}^{-}(\bar{z}))}+[V_{0}]_{C^{1/2,1}(Q_{R/24}^{-}(\bar{z}))})\nonumber\\
&\leq NrR^{-1-{(d+2)}/{p'}}\|Dv_{0}\|_{L_{p'}(Q_{R/12}^{-}(\bar{z}))}\nonumber\\
&\leq NrR^{-1-{(d+2)}/{p'}}\|h\|_{L_{p'}((Q_{2R}(\bar{z})\setminus Q_{R}(\bar{z}))\cap Q_{1/2}^{-})}.
\end{align}
Substituting the above estimate \eqref{est Dv0 V0} into \eqref{esti Dv h} and using the duality and H\"{o}lder's inequality, we have
\begin{align}\label{dilation Dv}
\|Dv_{1}\|_{L_{1}((Q_{2R}(\bar{z})\setminus Q_{R}(\bar{z}))\cap Q_{1/2}^{-})}\leq NrR^{-1}\|\hat{b}\|_{L_{1}(Q_{r}^{-}(\bar{z})\cap Q_{1/2}^{-})}.
\end{align}
Let $N_{0}$ be the smallest positive integer such that $Q_{1/2}^{-}\subset Q_{2^{N_{0}}cr}(\bar{z})$. By taking $R=cr, 2cr,\ldots,2^{N_{0}-1}cr$ in \eqref{dilation Dv} and summarizing, we obtain
\begin{align*}
\int_{Q_{1/2}^{-}\setminus Q_{cr}(\bar{z})}|Dv_{1}|\ dx\ dt&\leq N\sum_{k=1}^{N_{0}}2^{-k}\|\hat{b}\|_{L_{1}(Q_{r}^{-}(\bar{z})\cap Q_{1/2}^{-})}\leq N\int_{Q_{r}^{-}(\bar{z})\cap Q_{1/2}^{-}}|\hat{b}|\ dx\ dt.
\end{align*}
Therefore, $S$ satisfies the hypothesis of Lemma \ref{lemma weak}. The proof of this lemma is finished.
\end{proof}

Denote
$$
\phi(z_{0},r):=\inf_{\mathbf q\in\mathbb R^{n\times d}}\left(\fint_{Q_{r}^{-}(z_{0})}|(D_{x'}u,U)-\mathbf q|^{q}\ dx\ dt\right)^{1/q},$$
where $0<q<1$ is some fixed exponent. We are going to use Lemma \ref{weak est barv} to prove an iteration formula about the function $\phi(z_{0},r)$, from which we can derive the following
\begin{lemma}\label{lemma itera}
For any $\gamma\in (0,1)$ and $0<\rho\leq r\leq 1/4$, we have
\begin{align}\label{est phi'}
\phi(z_{0},\rho)\leq N\Big(\frac{\rho}{r}\Big)^{\gamma}r^{-d-2}\|(D_{x'}u,U)\|_{L_{1}(Q_{r}^{-}(z_{0}))}
+N\tilde{\omega}_{A}(\rho)\|Du\|_{L^{\infty}(Q_{r}^{-}(z_{0}))}+N\tilde{\omega}_{g}(\rho),
\end{align}
where $N=N(n,d,p,\nu,\gamma)$, and $\tilde\omega_{\bullet}(t)$ is a Dini function derived from $\omega_{\bullet}(t)$.
\end{lemma}

\begin{proof}
We apply Lemma \ref{weak est barv} with
$$
F=(\bar{A}(z'_{0},x^{d})-A(t,x))Du+g(t,x)-\bar{g}(z'_{0},x^{d}),
$$
\eqref{est A}, and follow the same argument as in deriving \cite[(3.7)]{dx} to obtain that
\begin{align}\label{holder v bar}
\left(\fint_{Q_{r/2}^{-}(z_{0})}|D_{x'}v|^{q}\ dx\ dt+\fint_{Q_{r/2}^{-}(z_{0})}|V|^{q}\ dx\ dt\right)^{{1}/{q}}\leq N\Big(\bar\omega_{A}(r)\|Du\|_{L_{\infty}(Q_{r}^{-}(z_{0}))}+\bar\omega_{g}(r)\Big),
\end{align}
where $\bar\omega_{\bullet}(r)=\omega_{\bullet}(r)+\hat\omega_1(r)$ and $V=\bar{A}^{d\beta}(z'_{0},x^{d})D_{\beta}v(t,x)$.

We now claim that
\begin{align}\label{iteration phi}
\phi(z_{0},\kappa^{j}r)\leq\kappa^{j\gamma}\phi(z_{0},r)+N\|Du\|_{L_{\infty}(Q_{r}^{-}(z_{0}))}\tilde\omega_{A}(\kappa^{j}r)+N\tilde\omega_{g}(\kappa^{j}r),
\end{align}
where $\kappa\in(0,1/2)$ is some fixed constant and
\begin{align}\label{tilde phi}
\tilde\omega_{\bullet}(t)=\sum_{i=1}^{\infty}\kappa^{i\gamma}\Big(\bar\omega_{\bullet}(\kappa^{-i}t)\chi_{\kappa^{-i}t\leq1}+\bar\omega_{\bullet}(1)\chi_{\kappa^{-i}t>1}\Big).
\end{align}
Furthermore, $\tilde\omega_{\bullet}(t)$ is a Dini function (see Lemma 1 in \cite{d}) and satisfies \eqref{equivalence}. Then for any $\rho$ satisfying $0<\rho\leq r\leq 1/4$, we take $j$ to be the integer with $\kappa^{j+1}<{\rho}/{r}\leq\kappa^{j}$. By using  \eqref{iteration phi} and \eqref{equivalence}, we have
\begin{align}\label{itera phi}
\phi(z_{0},\rho)\leq N\Big(\frac{\rho}{r}\Big)^{\gamma}\phi(z_{0},r)+N\tilde{\omega}_{A}(\rho)\|Du\|_{L_{\infty}(Q_{r}^{-}(z_{0}))}+N\tilde{\omega}_{g}(\rho),
\end{align}
where, it follows from H\"{o}lder's inequality that
\begin{align}\label{est phi}
\phi(z_{0},r)\leq\left(\fint_{Q_{r}^{-}(z_{0})}|(D_{x'}u,U)|^{q}\ dx\ dt\right)^{1/q}\leq Nr^{-d-2}\|(D_{x'}u,U)\|_{L_{1}(Q_{r}^{-}(z_{0}))}.
\end{align}
Combining \eqref{est phi} and \eqref{itera phi}, we get \eqref{est phi'}.

Finally, we prove the claim \eqref{iteration phi}. Let
\begin{equation}\label{defw u1}
u_{1}(x^{d})=\int_{x_{0}^{d}}^{x^{d}}(\bar{A}^{dd}(z'_{0},s))^{-1}\bar{g}_{d}(z'_{0},s)\,ds,
\quad \bar{u}=u-u_{1},\quad w=\bar{u}-v.
\end{equation}
Then a direct calculation yields $\bar{\ \mathcal{P}_{z'_{0}}}w=0$ in $Q_{r/2}^{-}(z_{0})$. For any $\kappa\in\big(0,1/4\big)$, by Lemma \ref{lemma xn} with a suitable scaling, we have
\begin{align}\label{DW kappa}
&\|D_{x'}w-(D_{x'}w)_{Q_{\kappa r}^{-}(z_{0})}\|_{L_{q}(Q_{\kappa r}^{-}(z_{0}))}^{q}+\|W-(W)_{Q_{\kappa r}^{-}(z_{0})}\|_{L_{q}(Q_{\kappa r}^{-}(z_{0}))}^{q}\nonumber\\
&\leq N(\kappa r)^{d+2+q}\left([D_{x'}w]_{C^{1/2,1}(Q_{r/4}^{-}(z_{0}))}^{q}
+[W]_{C^{1/2,1}(Q_{r/4}^{-}(z_{0}))}^{q}\right)\nonumber\\
&\leq N\kappa^{d+2+q}\int_{Q_{r/2}^{-}(z_{0})}|Dw|^{q}\ dx\ dt\nonumber\\
&\leq N\kappa^{d+2+q}\int_{Q_{r/2}^{-}(z_{0})}|(D_{x'}w,W)|^{q}\ dx\ dt,
\end{align}
where $W=\bar{A}^{d\beta}(z'_{0},x^{d})D_{\beta}w$. Define
\begin{equation*}
h(x^{d}):=\int_{0}^{x^{d}}\Big(\bar{A}^{dd}(z'_{0},s)\Big)^{-1}
\Big(q_{d}-\sum_{\beta=1}^{d-1}\bar{A}^{d\beta}(z'_{0},s)q_{\beta}\Big)\ ds,\quad \mathbf q=(q',q_{d})\in \mathbb R^{n\times d},
\end{equation*}
and
\begin{align*}
\tilde{w}:=w-\sum_{\beta=1}^{d-1}x^{\beta}q_{\beta}-h(x^{d}).
\end{align*}
Then
$$D_{x'}\tilde{w}=D_{x'}w-q',\quad \tilde{W}:=\bar{A}^{d\beta}(z'_{0},x^{d})D_{\beta}\tilde{w}=W-q_{d}.$$
Moreover, $\bar{\ \mathcal{P}_{z'_{0}}}\tilde{w}=0$ in $Q_{r/2}^{-}(z_{0})$. Now replacing $w$ and $W$ with $\tilde{w}$ and $\tilde{W}$ in \eqref{DW kappa}, respectively, we get
\begin{align*}
&\|D_{x'}w-(D_{x'}w)_{Q_{\kappa r}^{-}(z_{0})}\|_{L_{q}(Q_{\kappa r}^{-}(z_{0}))}^{q}+\|W-(W)_{Q_{\kappa r}^{-}(z_{0})}\|_{L_{q}(Q_{\kappa r}^{-}(z_{0}))}^{q}\nonumber\\
&\leq N\kappa^{d+2+q}\int_{Q_{r/2}^{-}(z_{0})}|(D_{x'}w-q',W-q_{d})|^{q}\ dx\ dt\\
&= N\kappa^{d+2+q}\int_{Q_{r/2}^{-}(z_{0})}|(D_{x'}w,W)-\mathbf q|^{q}\ dx\ dt,
\end{align*}
which implies
\begin{align}\label{holder w bar}
&\left(\fint_{Q_{\kappa r}^{-}(z_{0})}|D_{x'}w-(D_{x'}w)_{Q_{\kappa r}^{-}(z_{0})}|^{q}\ dx\ dt+\fint_{Q_{\kappa r}^{-}(z_{0})}|W-(W)_{Q_{\kappa r}^{-}(z_{0})}|^{q}\ dx\ dt\right)^{{1}/{q}}\nonumber\\
&\leq N_{0}\kappa\left(\fint_{Q_{r/2}^{-}(z_{0})}|(D_{x'}w,W)-\mathbf{q}|^{q}\ dx\ dt\right)^{{1}/{q}},
\end{align}
where $N_{0}=N_{0}(n,d,p,\nu,\Lambda)$.
Recalling that $\bar{u}=w+v$, we obtain from \eqref{holder w bar} that
\begin{align}\label{iteration u bar}
&\left(\fint_{Q_{\kappa r}^{-}(z_{0})}|D_{x'}\bar{u}-(D_{x'}w)_{Q_{\kappa r}^{-}(z_{0})}|^{q}+|\bar{U}-(W)_{Q_{\kappa r}^{-}(z_{0})}|^{q}\ dx\ dt\right)^{{1}/{q}}\nonumber\\
&\leq2^{{1}/{q}-1}\left(\fint_{Q_{\kappa r}^{-}(z_{0})}|D_{x'}w-(D_{x'}w)_{Q_{\kappa r}^{-}(z_{0})}|^{q}+|W-(W)_{Q_{\kappa r}^{-}(z_{0})}|^{q}\ dx\ dt\right)^{{1}/{q}}\nonumber\\
&\quad+N\left(\fint_{Q_{\kappa r}^{-}(z_{0})}|D_{x'}v|^{q}+|V|^{q}\ dx\ dt\right)^{{1}/{q}}\nonumber\\
&\leq N_{0}\kappa\left(\fint_{Q_{r/2}^{-}(z_{0})}|(D_{x'}\bar{u},\bar{U})-\mathbf{q}|^{q}\ dx\ dt\right)^{{1}/{q}}+N\kappa^{-(d+2)/{q}}
\left(\fint_{Q_{r/2}^{-}(z_{0})}|D_{x'}v|^{q}+|V|^{q}\ dx\ dt\right)^{{1}/{q}},
\end{align}
where $\bar{U}=\bar{A}^{d\beta}(z'_{0},x^{d})D_{\beta}\bar{u}$. Recalling that
\begin{equation*}
D_{x'}\bar{u}=D_{x'}u,\quad
U=A^{d\beta}(t,x)D_{\beta}u-g_{d}(t,x),\quad \text{and}\quad \bar{U}=\bar{A}^{d\beta}(z'_{0},x^{d})D_{\beta}u-\bar{g}_{d}(z'_{0},x^{d}),
\end{equation*}
we have for $z\in Q_{r}^{-}(z_{0})$,
\begin{align*}
|U-\bar{U}|\leq \|Du\|_{L_{\infty}(Q_{r}^{-}(z_{0}))}|A(z)-\bar{A}(z'_{0},x^{d})|+|g_{d}(z)-\bar{g}_{d}(z'_{0},x^{d})|.
\end{align*}
Thus, substituting \eqref{est A} and \eqref{holder v bar} into \eqref{iteration u bar}, we have
\begin{align}\label{formula D'u U}
&\left(\fint_{Q_{\kappa r}^{-}(z_{0})}\big|(D_{x'}u,U)-\big((D_{x'}w)_{Q_{\kappa r}^{-}(z_{0})},(W)_{Q_{\kappa r}^{-}(z_{0})}\big)\big|^{q}\ dx\ dt\right)^{1/q}\nonumber\\
&\leq N_{0}\kappa\left(\fint_{Q_{r}^{-}(z_{0})}|(D_{x'}u,U)-\mathbf{q}|^{q}\ dx\ dt\right)^{1/q}+N\kappa^{-(d+2)/q}
\left(\fint_{Q_{r}^{-}(z_{0})}|U-\bar{U}|^{q}\ dx\ dt\right)^{1/q}\nonumber\\
&\quad+N\kappa^{-(d+2)/q}
\left(\fint_{Q_{r/2}^{-}(z_{0})}|D_{x'}v|^{q}+|V|^{q}\ dx\ dt\right)^{1/q}\nonumber\\
&\leq N_{0}\kappa\left(\fint_{Q_{r}^{-}(z_{0})}|(D_{x'}u,U)-\mathbf{q}|^{q}\ dx\ dt\right)^{1/q}+N\kappa^{-(d+2)/q}\Big(\|Du\|_{L_{\infty}(Q_{r}^{-}(z_{0}))}\nonumber\\
&\quad\cdot\fint_{Q_{r}^{-}(z_{0})}|A(z)-\bar{A}(z'_{0},x^{d})|\ dx\ dt+\fint_{Q_{r}^{-}(z_{0})}|g_{d}(z)-\bar{g}_{d}(z'_{0},x^{d})|\ dx\ dt\Big)\nonumber\\
&\quad+N\kappa^{-(d+2)/q}
\left(\fint_{Q_{r/2}^{-}(z_{0})}|D_{x'}v|^{q}+|V|^{q}\ dx\ dt\right)^{1/q}\nonumber\\
&\leq N_{0}\kappa\left(\fint_{Q_{r}^{-}(z_{0})}|(D_{x'}u,U)-\mathbf{q}|^{q}\ dx\ dt\right)^{1/q}+N\kappa^{-(d+2)/q}\Big(\|Du\|_{L_{\infty}(Q_{r}^{-}(z_{0}))}\bar\omega_{A}(r)+\bar\omega_{g}(r)\Big).
\end{align}
Since $\mathbf{q}\in\mathbb R^{n\times d}$ is arbitrary, we obtain
\begin{align*}
\phi(z_{0},\kappa r)\leq N_{0}\kappa\phi(z_{0},r)+N\kappa^{-(d+2)/q}\Big(\|Du\|_{L_{\infty}(Q_{r}^{-}(z_{0}))}\bar\omega_{A}(r)+\bar\omega_{g}(r)\Big).
\end{align*}
For any given $\gamma\in(0,1)$, fix a $\kappa\in(0,1/2)$ sufficiently small so that $N_{0}\kappa\leq\kappa^{\gamma}$. We henceforth have
\begin{align*}
\phi(z_{0},\kappa r)\leq \kappa^{\gamma}\phi(z_{0},r)+N\Big(\|Du\|_{L_{\infty}(Q_{r}^{-}(z_{0}))}\bar\omega_{A}(r)+\bar\omega_{g}(r)\Big).
\end{align*}
By iteration and $\kappa^{\gamma}<1$, we obtain for $j=1,2,\ldots$,
\begin{align*}
\phi(z_{0},\kappa^{j}r)
&\leq\kappa^{j\gamma}\phi(z_{0},r)\\
&\quad+N\left(\|Du\|_{L^{\infty}(Q_{r}^{-}(z_{0}))}\sum_{i=1}^{j}\kappa^{(i-1)\gamma}\bar\omega_{A}(\kappa^{j-i}r)+\sum_{i=1}^{j}\kappa^{(i-1)\gamma}\bar\omega_{g}(\kappa^{j-i}r)\right).
\end{align*}
This gives \eqref{iteration phi}. The lemma is proved.
\end{proof}

Once we get Lemma \ref{lemma itera}, we can obtain the local boundedness of $Du$ in Lemma \ref{lem3.4} below. The proof of it is the same as that of \cite[Lemma 3.4]{dx} and thus omitted.
\begin{lemma}\label{lem3.4}
We have
\begin{align}\label{est Du''}
\|Du\|_{L_{\infty}(Q_{1/4}^{-})}\leq N\|(D_{x'}u,U)\|_{L_{1}(Q_{3/4}^{-})}+N\left(\int_{0}^{1}\frac{\tilde\omega_{g}(s)}{s}\ ds+\|g\|_{L_{\infty}(\cQ)}\right),
\end{align}
where $N>0$ is a constant depending only on $n,d,p,\nu,\gamma$, $\omega_{A}$, and $\hat\omega_{1}$.
\end{lemma}

\subsection{\bf Proof of Proposition \ref{main prop}} 
\begin{proof}
We recall that for each $z_0$, the coordinate system is chosen according to it. The proof is similar to that in \cite{dx}, so we only list the main differences. We claim that for {\em a.e.} $z_{0}\in Q_{3/4}^{-}$,
\begin{align}\label{est Du q}
&|(D_{x'}u(z_{0}),U(z_{0}))-\mathbf q_{z_{0},r}|\nonumber\\
&\leq N\left(\phi(z_{0},r)+\|Du\|_{L_{\infty}(Q_{r}^{-}(z_{0}))}\int_{0}^{r}\frac{\tilde\omega_{A}(s)}{s}\ ds+\int_{0}^{r}\frac{\tilde\omega_{g}(s)}{s}\ ds\right),
\end{align}
where $\mathbf q_{z_{0},r}\in\mathbb R^{n\times d}$ satisfying
$$
\phi(z_{0},r)=\left(\fint_{Q_{r}^{-}(z_{0})}|(D_{x'}u,U)-\mathbf q_{z_{0},r}|^{q}\ dx\ dt\right)^{1/q}.
$$
Note that \eqref{est Du q} is similar to \cite[(3.16)]{dx}. One can prove it by iteration, \eqref{iteration phi}, the assumption that $u\in C^{0,1}(Q_{3/4}^{-})$, and Lemma \ref{lemma omiga}. Then for $0<r<1/8$,
\begin{align*}
&\sup_{z_{0}\in Q_{1/8}^{-}}|(D_{x'}u(z_{0}),U(z_{0}))-\mathbf q_{z_{0},r}|\\
&\leq N\sup_{z_{0}\in Q_{1/8}^{-}}\phi(z_{0},r)+N\|Du\|_{L_{\infty}(Q_{1/4}^{-})}\int_{0}^{r}\frac{\tilde\omega_{A}(s)}{s}\ ds+N\int_{0}^{r}\frac{\tilde\omega_{g}(s)}{s}\ ds\\
&=:N\psi(r),
\end{align*}
where, it follows from Lemma \ref{lemma itera} that for any $0<r<1/8$,
\begin{align}\label{sup phi}
\sup_{z_{0}\in Q_{1/8}^{-}}\phi(z_{0},r)\leq N\left(r^{\gamma}\|(D_{x'}u,U)\|_{L_{1}(Q_{1/4}^{-})}+\tilde{\omega}_{A}(r)\|Du\|_{L_{\infty}(Q_{1/4}^{-})}+\tilde{\omega}_{g}(r)\right).
\end{align}

Now suppose that $z_{1}=(t_1,x_1)\in Q_{1/8}^{-}\cap \cQ_{j_{1}}$ for some $j_{1}\in[1,l+1]$. If $|z_{0}-z_{1}|_{p}\geq 1/32$, then by $$|(D_{x'}u(z_{0}),U(z_{0}))-(D_{x'}u(z_{1}),U(z_{1}))|\leq2\big(\|Du\|_{L_{\infty}(Q_{1/4}^{-})}+\|g\|_{L_{\infty}(\cQ)}\big)$$
and \eqref{est Du''}, we have
\begin{align}\label{C1 est2}
&|(D_{x'}u(z_{0}),U(z_{0}))-(D_{x'}u(z_{1}),U(z_{1}))|\nonumber\\
&\leq
N|z_{0}-z_{1}|_{p}^{\gamma}\left( \|(D_{x'}u,U)\|_{L_{1}(Q_{3/4}^{-})}+\int_{0}^{1}\frac{\tilde\omega_{g}(s)}{s}\ ds+\|g\|_{L_{\infty}(\cQ)}\right),
\end{align}
where $\gamma\in(0,1)$ is a constant. If $|z_{0}-z_{1}|_{p}<1/32$, we set $r=|z_{0}-z_{1}|_{p}$ and  claim that $\mbox{dist}(z_{0},\partial_{p}\cQ_{j_{0}}\cap\{t=t_{0}\})$ and $\mbox{dist}(z_{0},\partial_{p}\cQ_{j_{0}})$ are comparable. Indeed, on one hand, clearly
$$\mbox{dist}(z_{0},\partial_{p}\cQ_{j_{0}}\cap\{t=t_{0}\})
\geq\mbox{dist}(z_{0},\partial_{p}\cQ_{j_{0}}).$$
On the other hand, we may suppose that
$$\mbox{dist}(z_{0},\partial_{p}\cQ_{j_{0}}\cap\{t=t_{0}\})=|z_{0}-(t_{0},x'_{0},h_{j_{0}}(t_{0},x'_{0}))|_{p}$$
and
$$
\mbox{dist}(z_{0},\partial_{p}\cQ_{j_{0}})=|z_{0}-(t,x',h_{j_{0}}(t,x'))|_{p}.
$$
Then by using the triangle inequality and $h_{j_{0}}\in C^{\gamma_{0}}$ with $\gamma_{0}>1/2$, we have
\begin{align*}
&|z_{0}-(t_{0},x'_{0},h_{j_{0}}(t_{0},x'_{0}))|_{p}\\
&\leq |z_{0}-(t,x',h_{j_{0}}(t,x'))|_{p}+|(t-t_0,x'-x'_0,h_{j_{0}}(t,x')-h_{j_{0}}(t_{0},x'_{0}))|_{p}\\
&\leq |z_{0}-(t,x',h_{j_{0}}(t,x'))|_{p}+N(|t-t_{0}|^{1/2}+|x'-x'_0|+|t-t_{0}|^{\gamma_{0}})\\
&\leq N\mbox{dist}(z_{0},\partial_{p}\cQ_{j_{0}}).
\end{align*}
Now we continue the proof by discussing the following two cases.

{\bf Case 1.} If
$$
r>1/16\max\{\mbox{dist}(z_{0},\partial_{p}\cQ_{j_{0}}\cap\{t=t_{0}\}),
\mbox{dist}(z_{1},\partial_{p}\cQ_{j_{1}}\cap\{t=t_{1}\})\},
$$
then without loss of generality, we assume that $z_{0}$ is above $z_{1}$. By the triangle inequality, we have for $\forall~z\in Q_{r}^{-}(z_{1})$,
\begin{equation}
\begin{split}
\label{case2}
&|(D_{x'}u(z_{0}),U(z_{0}))-(D_{x'}u(z_{1}),U(z_{1}))|^{q}\\
&\leq|(D_{x'}u(z_{0}),U(z_{0}))-\mathbf q_{z_{0},2r}|^{q}+|\mathbf q_{z_{0},2r}-\mathbf q_{z_{1},2r}|^{q}+|(D_{y'}u(z_{1}),\tilde U(z_{1}))-\mathbf q_{z_{1},2r}|^{q}\\
&\quad+|(D_{y'}u(z_{1}),\tilde U(z_{1}))-(D_{x'}u(z_{1}),U(z_{1}))|^{q}\\
&\leq N\psi^{q}(2r)+|(D_{x'}u(z),U(z))-\mathbf q_{z_{0},2r}|^{q}+|(D_{y'}u(z),\tilde U(z))-\mathbf q_{z_{1},2r}|^{q}\\
&\quad+|(D_{y'}u(z),\tilde U(z))-(D_{x'}u(z),U(z))|^{q}+|(D_{y'}u(z_{1}),\tilde U(z_{1}))-(D_{x'}u(z_{1}),U(z_{1}))|^{q},
\end{split}
\end{equation}
where $D_{y'}$ denotes the first derivatives with respect to the first $d-1$ space variables in the coordinate system associated with $z_1$ and in this coordinate system, we use $D_{y}$ to define $\tilde U$. For the last term, one can see that $$
D_{x'}u(z_{1})-D_{y'}u(z_{1})=(D_{x'}u(z_{1}),D_{x^{d}}u(z_{1}))(I-X^{-1})I_{0},
$$ where $I_{0}=(I^{\alpha\beta})$ is a $d\times (d-1)$ matrix with
\begin{equation*}
I^{\alpha\beta}=\delta_{\alpha\beta}\ \ \mbox{for}~\alpha,\beta\in\{1,\dots,d-1\};\quad I^{d\beta}=0\ \ \mbox{for}~\beta\in\{1,\dots,d-1\},
\end{equation*}
$X=(X^{\alpha\beta})$ is a $d\times d$ matrix with
\begin{equation*}
X^{\alpha\beta}=\frac{\partial y^{\alpha}}{\partial x^{\beta}} \,\,~\mbox{for} ~\alpha,\beta=1,\dots,d,
\end{equation*}
and $I$ is a $d\times d$ identity matrix. We henceforth need to estimate $I-X^{-1}$. To end this, we suppose that for the fixed $t_1$, the closest point on $\partial_{p}Q_{j_{1}}\cap\{t=t_{1}\}$ to $z_{1}=(t_{1},x'_{1},x^d)$ is $(z'_{1},h_{j_{1}}(z'_{1}))$, and let $$n_{2}=\frac{\big(-\nabla_{x'}h_{j_{1}}(z'_{1}),1\big)^{\top}}{\sqrt{1+|\nabla_{x'}h_{j_{1}}(z'_{1})|^{2}}}$$
be the unit normal vector at $(z'_{1},h_{j_{1}}(z'_{1}))$ on the surface $\{(t_1,x',x^d): x^d=h_{j_{1}}(t_{1},x')\}$. The corresponding tangential vectors are given by
\begin{align*}
\tau_{2,1}=(1,0,\ldots,0,D_{x^{1}}h_{j_{1}}(z'_{1}))^{\top},\dots,
\tau_{2,d-1}=(0,0,\ldots,1,D_{x^{d-1}}h_{j_{1}}(z'_{1}))^{\top},
\end{align*}
from which we can use the Gram-Schmidt process to find an orthonormal basis $\{\hat\tau_{2,1},\ldots,\hat\tau_{2,d-1}\}$ of the tangent space.
Similarly, we denote
$$
n_{1}=\frac{\big(-\nabla_{x'}h_{j_{0}}(z'_{0}),1\big)^{\top}}
{\sqrt{1+|\nabla_{x'}h_{j_{0}}(z'_{0})|^{2}}}=(0',1)^{\top}
$$
to be the unit normal vector at $(z'_{0},h_{j_{0}}(z'_{0}))$, and the corresponding tangential vectors are
\begin{align*}
\tau_{1,1}=(1,0,\ldots,0)^{\top},
\ldots,\tau_{1,d-1}=(0,0,\ldots,1,0)^{\top}.
\end{align*}
It follows from the proof of Lemma \ref{volume} that  $|\nabla_{x'}h_{j_1}(z')|$ is bounded from above by $N\omega_1(r)$. Then we have
\begin{align*}
|n_{1}-n_{2}|&=\left|(0',1)^{\top}-\frac{\big(-\nabla_{x'}h_{j_{1}}(z'_{1}),1\big)^{\top}}{\sqrt{1+|\nabla_{x'}h_{j_{1}}(z'_{1})|^{2}}}\right|\\
&\leq N\omega_1(N_0|z_{0}-z_{1}|_{p})\leq N\omega_1(|z_{0}-z_{1}|_{p})\leq N\tilde\omega_1(|z_{0}-z_{1}|_{p}),
\end{align*}
where we used $\omega_{1}(N_{0}r)\leq N_{0}\omega_{1}(r)$ in the second inequality, which can be derived from the fact that $\omega_{0}$ is an increasing and concave function, $R$ is a monotonically increasing function with respect to $r$, and the definition of $\omega_1(r)=2\omega_0(2r+R)$ in the proof of Lemma \ref{volume}. This is also true for $|\tau_{1,i}-\tilde{\tau}_{2,i}|, i=1,\ldots,d-1$. We thus obtain
\begin{align*}
|D_{x'}u(z_{1})-D_{y'}u(z_{1})|\leq N\|Du\|_{L_{\infty}(Q_{1/4}^{-})}\tilde\omega_1(|z_{0}-z_{1}|_{p}).
\end{align*}
We similarly can estimate the difference of $U$ in different coordinate systems. Hence, we obtain
\begin{align}\label{diffe coor}
|(D_{x'}u(z_{1}),U(z_{1}))-(D_{y'}u(z_{1}),\tilde U(z_{1}))|\leq N\|Du\|_{L_{\infty}(Q_{1/4}^{-})}\tilde\omega_1(|z_{0}-z_{1}|_{p}).
\end{align}
Also, \eqref{diffe coor} is satisfied by the penultimate term of \eqref{case2}. Coming back to \eqref{case2}, we take the average over $z\in Q_{r}^{-}(z_{1})$ and take the $q$-th root to get
\begin{align*}
&|(D_{x'}u(z_{0}),U(z_{0}))-(D_{x'}u(z_{1}),U(z_{1}))|\nonumber\\
&\leq N\Big(\psi(2r)+\phi(z_{0},2r)+\phi(z_{1},2r)+\|Du\|_{L_{\infty}(Q_{1/4}^{-})}\tilde\omega_1(|z_{0}-z_{1}|_{p})\Big)\nonumber\\
&\leq N\Big(\psi(2r)+\|Du\|_{L_{\infty}(Q_{1/4}^{-})}\tilde\omega_1(|z_{0}-z_{1}|_{p})\Big).
\end{align*}
Therefore, it follows from \eqref{est Du''},  \eqref{sup phi}, and \eqref{equivalence} that
\begin{align}\label{C1 est1}
&|(D_{x'}u(z_{0}),U(z_{0}))-(D_{x'}u(z_{1}),U(z_{1}))|\nonumber\\
&\leq N|z_{0}-z_{1}|_{p}^{\gamma}\|(D_{x'}u,U)\|_{L_{1}(Q_{3/4}^{-})}+N\int_{0}^{|z_{0}-z_{1}|_{p}}\frac{\tilde{\omega}_{g}(s)}{s}\ ds\nonumber\\
&\quad+N\int_{0}^{|z_{0}-z_{1}|_{p}}\frac{\tilde{\omega}_{A}(s)}{s}\ ds\cdot\left(\|(D_{x'}u,U)\|_{L_{1}(Q_{3/4}^{-})}+\int_{0}^{1}\frac{\tilde{\omega}_{g}(s)}{s}\ ds+\|g\|_{L_{\infty}(\cQ)}\right).
\end{align}

{\bf Case 2.} If
$$
r\leq 1/16\max\{\mbox{dist}(z_{0},\partial_{p}\cQ_{j_{0}}\cap\{t=t_{0}\}),
\mbox{dist}(z_{1},\partial_{p}\cQ_{j_{1}}\cap\{t=t_{1}\})\},
$$
then $j_{0}=j_{1}$. Then we follow the same arguments as in \cite[Case 1.]{dx} to obtain
\begin{align}\label{C1 est}
&|(D_{x'}u(z_{0}),U(z_{0}))-(D_{x'}u(z_{1}),U(z_{1}))|\nonumber\\
&\leq N|z_{0}-z_{1}|_{p}^{\gamma}\|(D_{x'}u,U)\|_{L_{1}(Q_{3/4}^{-})}+N\int_{0}^{|z_{0}-z_{1}|_{p}}\frac{\tilde{\omega}_{g}(s)}{s}\ ds\nonumber\\
&\quad+N\int_{0}^{|z_{0}-z_{1}|_{p}}\frac{\tilde{\omega}_{A}(s)}{s}\ ds
\left(\|(D_{x'}u,U)\|_{L_{1}(Q_{3/4}^{-})}+\int_{0}^{1}\frac{\tilde{\omega}_{g}(s)}{s}\ ds+\|g\|_{L_{\infty}(\cQ)}\right).
\end{align}
Thus, Proposition \ref{main prop} is proved.
\end{proof}

\subsection{Proof of Corollary \ref{coro loc}}\label{sec coro loc}
The proof is a modification of  \cite[Corollary 1.6]{dx}, which in turn is based on the approach in \cite{a,b}. By the Sobolev embedding theorem in the parabolic setting (see, for instance, \cite[Lemma 8.1]{k}), we have $u\in L_{\frac{d+2}{d+1}}(\cQ)$. Fix some $p\in(1,\frac{d+2}{d+1})$ such that $d+2<p'<\infty$, where $p'=p/(p-1)$, we next prove that $Du\in L_{p,\text{loc}}(\cQ)$. Let $h\in C_{c}^{\infty}(\cQ)$ and $v\in \mathcal{H}_{2}^{1}(\cQ)$ be the solution of
\begin{align}\label{prob adj v}
\begin{cases}
\mathcal{P}^{*}v=\Div h&\quad \mbox{in}~\cQ\\
v=0&\quad \mbox{on}~((-T,0]\times\partial\cD)\cup(\{0\}\times\overline{\cD}),
\end{cases}
\end{align}
where $\mathcal{P}^{*}$ is the adjoint operator of $\mathcal{P}$ defined by
$$\mathcal{P}^{*}v:=v_{t}+D_{\beta}
((A^{\alpha\beta})^\top D_{\alpha}v)-D_{\alpha}((\hat{B}^{\alpha})^\top v)-(B^{\alpha})^\top D_{\alpha}v+C^\top v.$$ Then by Theorem \ref{thm1}, we obtain $Dv\in L_{\infty}((-T+\varepsilon,0)\times\cD_{\varepsilon})$. By the $\mathcal{H}_{2}^{1}$-estimate and $p'>2$, we have
\begin{align}\label{bound v}
\|v\|_{\mathcal{H}_{2}^{1}(\cQ)}\leq N\|h\|_{L_{2}(\cQ)}\leq N\|h\|_{L_{p'}(\cQ)}.
\end{align}
By Lemma \ref{lem loc lq} and \eqref{bound v}, we have
\begin{align*}
\|v\|_{\mathcal{H}_{p'}^{1}((-T+\varepsilon,0)\times\cD_{\varepsilon})}\leq N\big(\|h\|_{L_{p'}(\cQ)}+\|v\|_{L_{2}(\cQ)}\big)\leq N\|h\|_{L_{p'}(\cQ)}.
\end{align*}
This together with Sobolev-Morrey theorem and $p'>d+2$ implies that
\begin{equation*}
\|v\|_{L_{\infty}((-T+\varepsilon,0)\times\cD_{\varepsilon})}\leq N\|h\|_{L_{p'}(\cQ)}.
\end{equation*}
Fix $\zeta\in C_{c}^{\infty}((-T+\varepsilon,0)\times\cD_{\varepsilon})$ with $\zeta\equiv1$ on $\cQ'\subset\subset (-T+\varepsilon,0)\times\cD_{\varepsilon}$. Then we use $\zeta u$ as a test function to \eqref{prob adj v} and obtain
\begin{align}\label{weak  v1}
&\int_{\cQ}-v_{t}u\zeta+(A^{\alpha\beta})^\top D_{\alpha}v\left(\zeta D_{\beta}u+uD_{\beta}\zeta\right)+(B^{\alpha})^\top D_{\alpha}vu\zeta\nonumber\\
&\qquad
-(\hat{B}^{\alpha})^\top v\left(\zeta D_{\alpha}u+uD_{\alpha}\zeta\right)-C^\top vu\zeta=\int_{\cQ}h_{\alpha}D_{\alpha}(u\zeta).
\end{align}
On the other hand, recalling that $u\in\mathcal{H}_{1}^{1}(\cQ)$ is a weak solution of \eqref{systems}, we choose $\zeta v$ as a test function and get
\begin{align}\label{weak  u'}
&\int_{\cQ}u_{t}\zeta v+A^{\alpha\beta}D_{\beta}u\left(\zeta D_{\alpha}v+vD_{\alpha}\zeta\right)+B^{\alpha}u\left(\zeta D_{\alpha}v+vD_{\alpha}\zeta\right)-\hat{B}^{\alpha}D_{\alpha}uv\zeta-Cuv\zeta\nonumber\\
&=\int_{\cQ}g_{\alpha}\left(\zeta D_{\alpha}v+vD_{\alpha}\zeta\right)-f\zeta v.
\end{align}
Combining \eqref{weak  v1} and \eqref{weak  u'}, we obtain
\begin{align*}
\int_{\cQ}h_{\alpha}D_{\alpha}(u\zeta)
&=\int_{\cQ}uv\zeta_{t}-\int_{\cQ}A^{\alpha\beta}vD_{\beta}uD_{\alpha}\zeta
+\int_{\cD}(A^{\alpha\beta})^\top uD_{\alpha}vD_{\beta}\zeta-uvB^{\alpha}D_{\alpha}\zeta\nonumber\\
&\quad-(\hat{B}^{\alpha})^\top uv D_{\alpha}\zeta
+\int_{\cQ}g_{\alpha}\left(\zeta D_{\alpha}v+vD_{\alpha}\zeta\right)-f\zeta v,
\end{align*}
which is similar to \cite[(4.8)]{dx}.
Then by replicating the argument in the proof of \cite[Corollary 1.6]{dx}, we have
\begin{align*}
\left|\int_{\cQ}h_{\alpha}D_{\alpha}(u\zeta)\right|\leq N\left(\|g\|_{L_{\infty}(\cQ)}+\|f\|_{L_{\infty}(\cQ)}+\|u\|_{\mathcal{H}_{1}^{1}(\cQ)}\right)\|h\|_{L_{p'}(\cQ)}
\end{align*}
for all $h\in C_{c}^{\infty}(\cQ)$. Hence, $u\in \mathcal{H}_{p}^{1}(\cQ')$ 
and
\begin{align*}
\|u\|_{\mathcal{H}_{p}^{1}(\cQ')}\leq N\left(\|g\|_{L_{\infty}(\cQ)}+\|f\|_{L_{\infty}(\cQ)}+\|u\|_{\mathcal{H}_{1}^{1}(\cQ)}\right).
\end{align*}
The corollary is proved.

\section{Proof of Theorem \ref{thm holder}}\label{section thm holder}
\subsection{The continuity of \texorpdfstring{$D_{x'}u$}{D'u} and \texorpdfstring{$U$}{U}}

We first prove \eqref{Lip Du}. Similar to the proof of Theorem \ref{thm1}, we take $z_{0}\in Q_{3/4}^{-}\cap \cQ_{j_{0}}$. Let $A^{(j)}\in C^{\delta/2,\delta}(\overline{\cQ}_{j})$, $1\leq j\leq l+1$, be matrix-valued functions, and $B^{(j)}, g^{(j)}$ be in $C^{\delta/2,\delta}(\overline{\cQ}_{j})$. Define the piecewise constant (matrix-valued) functions
\begin{align*}
\bar{A}(z)=
A^{(j)}(z_{0}),\ \ z\in\Omega_{j_0},\quad
\bar{A}(z)=
A^{(j)}(z'_{0},h_{j}(z'_{0})),\ \ z\in\Omega_{j},\ \ j\neq j_0.
\end{align*}
From $B^{(j)}$ and $g^{(j)}$, we similarly define piecewise constant functions $\bar{B}$ and $\bar{g}$. Notice that these functions only depend on the center $z_0$, but are independent of the radius of the cylinder $r$. Using Lemma \ref{volume}, we immediately get the following result.
\begin{lemma}\label{difference holder}
Let $A, \bar{A}, B, \bar{B}, g$, and $\bar{g}$ be defined as above, there exists a positive constant $N$, depending only on $d,l,\mu,\delta,\nu,\Lambda$, $\max_{1\leq j\leq l+1}\|A\|_{C^{\delta/2,\delta}(\overline{Q}_{j})}$, $\max_{1\leq j\leq l+1}\|B\|_{C^{\delta/2,\delta}(\overline{Q}_{j})}$, $\max_{1\leq j\leq l+1}\|g\|_{C^{\delta/2,\delta}(\overline{Q}_{j})}$ and $\max_{1\leq j\leq l+1}\|h_{j}\|_{C^{1,\mu}(\overline{D}_{j})}$, such that for $0<r\leq 1$,
\begin{align*}
\fint_{Q_{r}^{-}(z_{0})}|A-\bar{A}|\ dx\ dt+\fint_{Q_{r}^{-}(z_{0})}|B-\bar{B}|\ dx\ dt
+\fint_{Q_{r}^{-}(z_{0})}|g-\bar{g}|\ dx\ dt\leq Nr^{\delta'},
\end{align*}
where $\delta'=\min\{\delta,2\gamma_{0}-1\}$.
\end{lemma}
Thus, \eqref{Lip Du} directly follows from \eqref{C1 est}, \eqref{C1 est1}, and \eqref{C1 est2} by taking $\gamma\in(\delta',1)$.

Next, we observe from \eqref{Lip Du} that for each $j=1,\ldots,M$,
$$
D_{x'}u, U\in C^{\delta'/2,\delta'}(\overline{\cQ_{j}}\cap((-T+\varepsilon,0)\times\cD_{\varepsilon})).
$$
On the other hand, since
$$D_{d}u=(A^{dd})^{-1}\left(U+g_{d}-B^{d}u-\sum_{\beta=1}^{d-1}A^{d\beta}D_{\beta}u\right),$$
we conclude that $D_{d}u\in C^{\delta'/2,\delta'}(\overline{\cQ_{j}}\cap((-T+\varepsilon,0)\times\cD_{\varepsilon}))$.

\subsection{The estimate of \texorpdfstring{$\langle u\rangle_{1+\delta'}$}{u_t}}

The proof is again based on the Campanato's method, but we work on $u$ itself instead of its first derivatives. The key point is to prove that the mean oscillation of $u$ in cylinders vanishes in the order $r^{1+\delta'}$ as the radii $r$ of cylinders go to zero. In order to derive this, as  shown in Subsection \ref{subsec prb},  we only need to treat the case without lower-order terms and the data $f$.  Then we prove a weak type-$(1,1)$ estimate for solutions to parabolic systems with coefficients are of piecewise Dini mean oscillation. Finally, we introduce a set consisting of polynomials with respect to $x$  and use it to prove an estimate of the difference between $u$ and some polynomial in the $L_{q}$-mean sense, $q\in(0,1)$.

Fix $z_{0}\in \Big((-9/16,0)\times B_{3/4}\Big)\cap \cQ_{j_{0}}$ and take $0<r<R\leq1/4$, we take the coordinate system associated with $z_0$ and follow the proof of Theorem \ref{thm1}. As in Section \ref{proof thm1}, we denote
\begin{align}\label{operator barP}
\bar{\mathcal{P}}u:=-u_{t}+D_{\alpha}(\bar{A}^{\alpha\beta}(z'_{0},x^{d})D_{\beta}u).
\end{align}
Then
\begin{align*}
\bar{\mathcal{P}}u&=\Div(g+(\bar{A}(z'_{0},x^{d})-A(z))Du).
\end{align*}
Let $\tilde{\mathcal{P}\ }$ be the modified operator corresponding to $\bar{\mathcal{P}}$ as in Section \ref{proof thm1}.
Let $v\in \mathcal{H}_{p}^{1}(Q_{r}^{-}(z_{0}))$ be a weak solution to
\begin{align}\label{equation v0000}
\begin{cases}
\bar{\mathcal{P}\ }v=\Div(g-\bar{g}+(\bar{A}(z'_{0},x^{d})-A(z))Du)&\ \mbox{in}~Q_{r}^{-}(z_{0}),\\
v=0&\ \mbox{on}~\partial_{p}Q_{r}^{-}(z_{0}),
\end{cases}
\end{align}
where $\bar{g}:=\bar{g}(z'_{0},x^d)$ is the piecewise constant function corresponding to $g$ defined in Subsection \ref{subsection domain}. Next we give two lemmas which are the key ingredients of the proof of the estimate of $\langle u\rangle_{1+\delta'}$.

\begin{lemma}[Weak type-$(1,1)$ estimate]
            \label{weak est v}
Let $R\in (0,1/4)$ and $p\in(1,\infty)$. Let $v\in \mathcal{H}_{p}^{1}(Q_{R}^{-}(z_{0}))$ be a weak solution to the problem
\begin{align*}
\begin{cases}
\tilde{\mathcal{P}\ }v=\Div(F\chi_{Q_{R/2}^{-}(z_{0})})&\ \mbox{in}~Q_{R}^{-}(z_{0}),\\
v=0&\ \mbox{on}~\partial_p Q^-_{R}(z_{0}),\\
\end{cases}
\end{align*}
where $F\in L_{p}(Q_{R/2}^{-}(z_{0}))$. Then for any $s>0$, we have
\begin{align*}
|\{z\in Q_{R/2}^{-}(z_{0}): |v(z)|>s\}|\leq\frac{NR}{s}\|F\|_{L_{1}(Q_{R/2}^{-}(z_{0}))},
\end{align*}
where $N=N(n,d,p,\nu)$.
\end{lemma}

\begin{proof}
As in the proof of Lemma \ref{weak est barv}, we set $z_{0}=0$, $R=1$, $\bar{A}^{\alpha\beta}(x^{d}):=\bar{A}^{\alpha\beta}(0',x^{d})$, $\bar{\mathcal{P}\ }:=\bar{\ \mathcal{P}_{0'}}$ for simplicity, and follow the same notation there.  We are going to prove that the hypothesis of Lemma \ref{lemma weak} is satisfied. Set $c=24$ and fix $\bar{z}=(\bar{t},\bar{y})\in Q_{1/2}^{-}$, $0<r<1/4$. Let $\hat{b}\in L_{p}(Q_{1}^{-})$ be supported in $Q_{r}^{-}(\bar{z})\cap Q_{1/2}^{-}$ with mean zero, $b=E\hat{b}$, and $v_{1}\in \mathcal{H}_{p}^{1}(Q_{1}^{-})$ be the unique weak solution of
\begin{align*}
\begin{cases}
\tilde{\mathcal{P}\ }v_{1}=\Div b&\ \mbox{in}~Q_{1}^{-},\\
v_{1}=0&\ \mbox{on}~\partial_p Q_{1}^-.
\end{cases}
\end{align*}
For any $R\geq cr$ such that $Q_{1/2}^{-}\setminus Q_{R}(\bar{z})\neq\emptyset$ and $h\in C_{0}^{\infty}((Q_{2R}(\bar{z})\setminus Q_{R}(\bar{z}))\cap Q_{1/2}^{-})$, let $v_{0}\in \mathcal{H}_{p'}^{1}(Q_{1}^{-})$ be a weak solution of
\begin{align*}
\begin{cases}
\tilde{\mathcal{P}^{*}\ }v_{0}=h&\ \mbox{in}~Q_{1}^{-},\\
v_{0}=0&\ \mbox{on}~((-1,0]\times\partial B_{1})\cup (\{0\}\times \overline{B_1}),
\end{cases}
\end{align*}
where ${1}/{p}+{1}/{p'}=1$ and $\tilde{\mathcal{P}^{*}\ }$ is the adjoint operator of $\tilde{\mathcal{P}\ }$ defined by
$$\tilde{\mathcal{P}^{*}\ }u:=u_{t}+D_{\beta}(\tilde{A}^{\beta\alpha}D_{\alpha}u).$$
In view of the definition of weak solutions and the assumption of $\hat b$, we have
\begin{align*}
\int_{Q_{1/2}^{-}}v_{1}h&=\int_{Q_{1/2}^{-}}Dv_{0}\cdot b=\int_{Q_{r}^{-}(\bar{z})\cap Q_{1/2}^{-}}
\left(
D_{x'}v_{0}, V_{0}
\right)\cdot\hat{b}\\
&=\int_{Q_{r}^{-}(\bar{z})\cap Q_{1/2}^{-}}
\left(
D_{x'}v_{0}-D_{x'}v_{0}(\bar{z}), V_{0}-V_{0}(\bar{z})
\right)\cdot\hat{b},
\end{align*}
where
$V_{0}=\bar{A}^{d\beta}(x^{d})D_{\beta}v_{0}$. Hence, as before we have
\begin{align}\label{dilation Dvb}
\|v_{1}\|_{L_{1}(( Q_{2R}(\bar{z})\setminus Q_{R}(\bar{z}))\cap Q_{1/2}^{-})}\leq NrR^{-1}\|\hat{b}\|_{L_{1}(Q_{r}^{-}(\bar{z})\cap Q_{1/2}^{-})},
\end{align}
where we used \eqref{est Dv0 V0}, the duality, and H\"{o}lder's inequality. Let $N_{0}$ be the smallest positive integer such that $Q_{1/2}^{-}\subset Q_{2^{N_{0}}cr}(\bar{z})$. By taking $R=cr, 2cr,\ldots,2^{N_{0}-1}cr$ in \eqref{dilation Dvb} and summarizing, we obtain
\begin{align*}
\int_{Q_{1/2}^{-}\setminus Q_{cr}(\bar{z})}|v_{1}|\ dx\ dt&\leq N\sum_{k=1}^{N_{0}}2^{-k}\|\hat{b}\|_{L_{1}(Q_{r}^{-}(\bar{z})\cap Q_{1/2}^{-})}\leq N\int_{Q_{r}^{-}(\bar{z})\cap Q_{1/2}^{-}}|\hat{b}|\ dx\ dt.
\end{align*}
Therefore, the hypothesis of Lemma \ref{lemma weak} is satisfied. The proof of this lemma is finished.
\end{proof}

Denote
$$
\mathbb{P}_{1}=\bigg\{p:~p(x)=\sum_{\beta=1}^{d-1}\ell_{\beta}x^{\beta}+\vartheta(x^{d}) \bigg\},
$$
where $\ell_{\beta}$'s are constants and $\vartheta(\cdot)$ is a measurable function.
For any $z_{0}\in (-9/16,0)\times B_{3/4}$, we also denote
$$
\mathbb{P}^{z_0}_{1}=\left\{p:p(x)=\ell_0+\sum_{\beta=1}^{d-1}\ell_\beta (x^{\beta}-x_0^\beta)
+\int_{x_0^d}^{x^{d}}\Big(\bar{A}^{dd}(z'_{0},s)\Big)^{-1}
\Big(
\ell_{d}-\sum_{\beta=1}^{d-1}
\bar{A}^{d\beta}(z'_{0},s)\ell_{\beta}\Big)\ ds \right\},
$$
where $\ell_{\beta}$'s are constants. Clearly, $\mathbb{P}^{z_0}_{1}\subset \mathbb{P}_{1}$.
\begin{lemma} \label{lem3.13}
Let $r>0$ and $p\in \mathbb{P}^{z_0}_{1}$. Suppose that
\begin{equation}
                            \label{eq3.05}
\fint_{B_r(x_0)}|p(x)|^q\,dx\le C_0^q r^{q(1+\delta')},
\end{equation}
where $C_0\ge 0$ is a constant.
Then we have
$$
|\ell_0|\le NC_0 r^{1+\delta'},\quad |\ell_\beta|\le NC_0 r^{\delta'},\ \beta=1,\ldots,d,
$$
where $N>0$ depends only on $d$, $n$, $\nu$, $q$, and $\delta'$.
\end{lemma}
\begin{proof}
Without loss of generality, we may assume that $x_0=0$. Since $p(x)-p(-x^1,x^2,\ldots,x^d)=2\ell_1x^1$, we have
$$
\fint_{B_r}|2\ell_1 x^1|^q\,dx
\le \fint_{B_r}(|p(x)|^q+|p(-x^1,x^2,\ldots,x^d)|^q)\,dx
\le C_0^q r^{q(1+\delta')},
$$
which implies that $|\ell_1|\le NC_0 r^{\delta'}$. Similarly, we get
\begin{align}
                                \label{eq2.51}
|\ell_\beta|\le NC_0 r^{\delta'}\quad \text{for}\  \beta=2,\ldots,d-1.
\end{align}
By using the parabolicity condition of $\bar{A}^{dd}$, we have
$$
|p(x)-p(x',x^d/2)|\ge  (N^{-1}\nu |\ell_d|-N\nu^{-1}|q'|)|x^d|,
$$
where $q'=(\ell_1,\ldots,\ell_{d-1})$. This together with \eqref{eq2.51} gives
$$
\fint_{B_r}|\ell_d|^q |x^d|^q\,dx
\le NC_0^q r^{q(1+\delta')}+\fint_{B_r}|p(x)-p(x',x^d/2)|^q\,dx\le NC_0^q r^{q(1+\delta')},
$$
which implies
$$
|\ell_d|\le NC_0 r^{\delta'}.
$$
Finally, the bound of $\ell_0$ follows from \eqref{eq3.05} and the bounds of $\ell_\beta$, where $\beta=1,\ldots,d$. The lemma is proved.
\end{proof}

Now we are ready to give the proof of \eqref{angle u t}.
\begin{proof}[\bf Proof of \eqref{angle u t}.]
We divide the proof into three steps.

{\bf Step 1. Claim:} for any $z_{0}\in (-9/16,0)\times B_{3/4}$ and $r\in(1,1/4)$, in the coordinate system associated with $z_0$ we can find
$p^{r,z_0}=p^{r,z_0}(x)$ in the form
\begin{align*}
&\ell_0^{r,z_0}+\sum_{\beta=1}^{d-1}\ell_\beta^{r,z_0}(x^{\beta}-x_{0}^{\beta})
+\int_{x_{0}^{d}}^{x^{d}}\Big(\bar{A}^{dd}(z'_{0},s)\Big)^{-1}
\Big(\bar{g}_{d}(z'_{0},s)+
\ell_{d}^{r,z_0}-\sum_{\beta=1}^{d-1}
\bar{A}^{d\beta}(z'_{0},s)\ell_{\beta}^{r,z_0}\Big)\ ds\\
&\quad \in u_1+\mathbb{P}^{z_0}_{1},
\end{align*}
where $\ell_\beta^{r,z_0}$ are constants and $u_1$ is defined in \eqref{defw u1},
such that
\begin{equation} \label{eq12.21}
\fint_{Q_{r}^{-}(z_{0})}|u-p^{r,z_0}|^{q}
\leq NC_{0}^{q}r^{q(1+\delta')},
\end{equation}
where
\begin{equation}\label{def C0}
C_{0}=\sum_{j=1}^{M}|g|_{\delta/2,\delta;\overline{\cQ}_{j}}+\|u\|_{L_{p}(\cQ)}+\|Du\|_{L_{1}(\cQ)}.
\end{equation}

For simplicity, we assume that $x_{0}=0$. Applying Lemma \ref{weak est v} to \eqref{equation v0000} and using the same argument that led to \eqref{holder v bar},
we obtain
\begin{align}\label{holder v 00}
\left(\fint_{Q_{r/2}^{-}(z_{0})}|v|^{q}\ dx\ dt\right)^{{1}/{q}}\leq Nr^{1+\delta'}\Big(\|Du\|_{L_{\infty}(Q_{r}^{-}(z_{0}))}+\sum_{j=1}^{M}|g|_{\delta/2,\delta;\overline{\cQ}_{j}}\Big),
\end{align}
where $q\in(0,1)$. Recall that $w=u-u_{1}-v$ satisfies $\mathcal{\bar P}w=0$ in $Q_{r/2}^{-}(z_{0})$. Define
\begin{align*}
p_1(x)&=\hat{T}^{1}w=w(z_0)+x'\cdot D_{x'}w(z_{0})\\
&\qquad +\int_{0}^{x^{d}}\Big(\bar{A}^{dd}(z'_{0},s)\Big)^{-1}
\Big(W(z_{0})-\sum_{\beta=1}^{d-1}
\bar{A}^{d\beta}(z'_{0},s)D_{\beta}w(z_{0})\Big)\ ds\in \mathbb{P}^{z_0}_{1},
\end{align*}
where $W=\sum_{\beta=1}^{d}
\bar{A}^{d\beta}(z'_{0},x^{d})D_{\beta}w$.
Then we have
$$
\|D_{x'}(w-p_1)\|_{L_\infty(Q_{\kappa r}^{-}(z_{0}))}
=\|D_{x'}w-D_{x'}w(z_0)\|_{L_\infty(Q_{\kappa r}^{-}(z_{0}))}
\le N\kappa r[D_{x'}w]_{C^{1/2,1}(Q_{\kappa r}^{-}(z_{0}))}
$$
and
$$
\|\bar{A}^{d\beta}(z'_{0},x^d)D_\beta(w-p_1)\|_{L_\infty(Q_{\kappa r}^{-}(z_{0}))}
=\|W-W(z_0)\|_{L_\infty(Q_{\kappa r}^{-}(z_{0}))}
\le N\kappa r[W]_{C^{1/2,1}(Q_{\kappa r}^{-}(z_{0}))},
$$
which together with Lemma \ref{lemma xn} with a suitable scaling imply
$$
\|D(w-p_1)\|_{L_\infty(Q_{\kappa r}^{-}(z_{0}))}
\le N\kappa r^{-1-(d+2)/q}\|w\|_{L_q(Q_{r/2}^{-}(z_{0}))}.
$$
Since $(w-p_1)(z_0)=0$, by Lemma \ref{lemma xn} with a suitable scaling, we have
\begin{align}
                    \label{eq3.45}
&\fint_{Q_{\kappa r}^{-}(z_{0})}|w-p_1|^{q}\ dx\ dt
\le N\fint_{Q_{\kappa r}^{-}(z_{0})}(|w-w(t_0,x)|+|w(t_0,x)-p_1|^{q}\ dx\ dt\nonumber\\
&\leq N(\kappa r)^{2q}\|w_t\|^q_{L_\infty(Q_{\kappa r}^{-}(z_{0}))}
+N(\kappa r)^{q}\|D(w-p_1)\|_{L_\infty(Q_{\kappa r}^{-}(z_{0}))}^{q}\nonumber\\
&\leq N\kappa^{2q}\fint_{Q_{r/2}^{-}(z_{0})}|w|^{q}\ dx\ dt.
\end{align}
Noting that $w-p$ satisfies the same equation as $w$ for any $p\in \mathbb{P}^{z_0}_{1}$ and $\hat T^1 p=p$, we then infer from \eqref{eq3.45} that for any $p\in \mathbb{P}^{z_0}_{1}$,
\begin{equation}
                    \label{eq3.45b}
\fint_{Q_{\kappa r}^{-}(z_{0})}|w-p_1|^{q}\ dx\ dt
\leq N\kappa^{2q}\fint_{Q_{r/2}^{-}(z_{0})}|w-p|^{q}\ dx\ dt.
\end{equation}

By using \eqref{holder v 00}, \eqref{eq3.45b}, $w=u-u_1-v$, the triangle inequality, \eqref{est Du''}, and the proof of Lemma \ref{lemma itera} (cf. \eqref{iteration u bar} and \eqref{formula D'u U}), we have
\begin{align}
&\fint_{Q_{\kappa r}^{-}(z_{0})}|u-u_1-p_1|^{q}\ dx\ dt\nonumber\\
&\leq\fint_{Q_{\kappa r}^{-}(z_{0})}|w-p_1|^{q}\ dx\ dt+\fint_{Q_{\kappa r}^{-}(z_{0})}|v|^{q}\ dx\ dt\nonumber\\
&\leq\kappa^{2q}\fint_{Q_{r/2}^{-}(z_{0})}|w-p|^{q}\ dx\ dt+N\kappa^{-(d+2)}r^{q(1+\delta')}
\Big(\|Du\|_{L_{\infty}(Q_{r}^{-}(z_{0}))}+\sum_{j=1}^{M}|g|_{\delta/2,\delta;
\overline{\cQ}_{j}}\Big)^{q}\nonumber\\
            \label{eq9.54}
&\leq\kappa^{2q}\fint_{Q_{r/2}^{-}(z_{0})}
|u-u_1-p|^{q}\ dx\ dt+NC_{0}^{q}\kappa^{-(d+2)}r^{q(1+\delta')},
\end{align}
where $C_{0}$ is defined in \eqref{def C0}.

Denote
$$
F(r):=\inf_{p\in\mathbb{P}_{1}^{z_0}}
\fint_{Q_{r}^{-}(z_{0})}|u-u_1-p|^{q}\ dx\ dt.
$$
Then from \eqref{eq9.54} we have
\begin{align*}
F(\kappa r)\leq\kappa^{2q}F(r)+NC_{0}^{q}\kappa^{-(d+2)}r^{q(1+\delta')}
\end{align*}	
for any $r\in (0,1/4)$ and $\kappa\in(0,1/2)$. Then by using a well-known iteration argument (see, for instance, \cite[Lemma 2.1, p. 86]{g}), we have
\begin{align*}
F(r)\leq r^{q(1+\delta')}F(1/4)+NC_{0}^{q}r^{q(1+\delta')}\leq NC_{0}^{q}r^{q(1+\delta')}.
\end{align*}
Thus, we
conclude \eqref{eq12.21}.

{\bf Step 2. Convergence of $\ell_{\beta}^{r,z_{0}}$, $\beta=1,\dots,d$.}
By using the triangle inequality and \eqref{eq12.21}, we have
\begin{align*}
\fint_{Q_{r}^{-}(z_{0})}|p^{r,z_0}-p^{2r,z_{0}}|^{q}
&\le \fint_{Q_{r}^{-}(z_{0})}|u-p^{r,z_0}|^q+|u-p^{2r,z_0}|^{q}
\leq NC_{0}^{q}r^{q(1+\delta')}.
\end{align*}
Since
\begin{align*}
&p^{r,z_{0}}-p^{2r,z_{0}}
=\ell_0^{r,z_{0}}-\ell_0^{2r,z_{0}}+\sum_{\beta=1}^{d-1}\big(\ell_\beta^{r,z_{0}}-\ell_\beta^{2r,z_{0}}\big) x^{\beta}\\
&\quad+\int_{0}^{x^{d}}\Big(\bar{A}^{dd}(z'_{0},s)\Big)^{-1}
\Big(\ell_{d}^{r,z_{0}}-
\ell_{d}^{2r,z_{0}}-\sum_{\beta=1}^{d-1}
\bar{A}^{d\beta}(z'_{0},s)\big(\ell_{\beta}^{r,z_{0}}-\ell_{\beta}^{2r,z_{0}}\big)\Big)\ ds\in\mathbb{P}_{1}^{z_0},
\end{align*}
we can apply Lemma \ref{lem3.13} to get
\begin{equation}\label{difference ell}
|\ell_0^{r,z_0}-\ell_0^{2r,z_0}|\le NC_0 r^{1+\delta'},\quad |\ell^{r,z_0}_\beta-\ell^{2r,z_0}_\beta|\le NC_0 r^{\delta'},\ \beta=1,\ldots,d.
\end{equation}
Therefore, the limits of $\ell^{r,z_0}_\beta$ exist and are denoted by $\ell_\beta^{z_0}$ for $\beta=0,1,\ldots,d$.
Since $u$ is continuous, it is easily seen  that for any $r\in (0,1/4)$,
\begin{equation}\label{limit l}
u(z_0)=\ell_0^{z_0},\quad |u(z_0)-\ell_0^{r,z_0}|\le NC_0 r^{1+\delta'}.
\end{equation}

Next, we claim that for $\beta=1,\dots,d-1$,
\begin{align}\label{difference Du U}
\fint_{Q_{r/2}^{-}(z_{0})}|D_{\beta}u-\ell_{\beta}^{r,z_{0}}|^{q}+\fint_{Q_{r/2}^{-}(z_{0})}|U-\ell_{d}^{r,z_{0}}|^{q}
\leq NC_{0}^{q}r^{q\delta'}.
\end{align}
In fact, recalling the definition of the operator $\bar{\mathcal{P}}$ in \eqref{operator barP}, we have
\begin{align*}
\bar{\mathcal{P}}(u-p^{r,z_{0}})=\Div(g-\bar{g}+(\bar{A}(z'_{0},x^{d})-A(z))Du).
\end{align*}
Then by using Lemma \ref{lem loc lq} with $0<q<1<p<\infty$, the interpolation inequality,  Lemma \ref{lem3.4}, \eqref{eq12.21}, and Lemma \ref{difference holder}, we have
\begin{align*}
&\|D(u-p^{r,z_{0}})\|_{L_{p}(Q_{r/2}^{-}(z_{0}))}\\
&\leq Nr^{-1+\frac{d+2}{p}-\frac{d+2}{q}}\|u-p^{r,z_{0}}\|_{L_{q}(Q_{r}^{-}(z_{0}))}
+N\|g-\bar{g}+(\bar{A}(z'_{0},x^{d})-A)Du\|_{L_{p}(Q_{r}^{-}(z_{0}))}\\
&\leq Nr^{-1+\frac{d+2}{p}-\frac{d+2}{q}}\|u-p^{r,z_{0}}\|_{L_{q}(Q_{r}^{-}(z_{0}))}
+NC_{0}r^{\frac{d+2}{p}-(d+2)}\|g-\bar{g}+(\bar{A}(z'_{0},x^{d})-A)Du\|_{L_{1}(Q_{r}^{-}(z_{0}))}\\
&\leq NC_{0}r^{\delta'+\frac{d+2}{p}}.
\end{align*}
By using H\"{o}lder's inequality, we obtain
\begin{align*}
\fint_{Q_{r/2}^{-}(z_{0})}|D(u-p^{r,z_{0}})|^{q}&\leq\left(\fint_{Q_{r/2}^{-}(z_{0})}|D(u-p^{r,z_{0}})|^{p}\right)^{\frac{q}{p}}\nonumber\\
&\leq NC_{0}^{q}r^{q\delta'+\frac{q(d+2)}{p}}\leq NC_{0}^{q}r^{q\delta'}.
\end{align*}
We thus obtain
\begin{align}\label{difference Du p}
\fint_{Q_{r/2}^{-}(z_{0})}|D_{x'}(u-p^{r,z_{0}})|^{q}+\fint_{Q_{r/2}^{-}(z_{0})}|\bar{A}^{d\beta}D_{\beta}(u-p^{r,z_{0}})|^{q}\leq NC_{0}^{q}r^{q\delta'}.
\end{align}
A direct calculation yields
$$D_{\beta}p^{r,z_{0}}=\ell_{\beta}^{r,z_{0}},~\beta=1,\dots,d-1,\quad \sum_{\beta=1}^{d}\bar{A}^{d\beta}D_{\beta}p^{r,z_{0}}=\bar{g}_{d}+\ell_{d}^{r,z_{0}}.$$
Combining $U=\sum_{\beta=1}^{d}A^{d\beta}D_{\beta}u-g_{d}$, \eqref{difference Du p}, the triangle inequality, and Lemma \ref{difference holder}, we have
\begin{align*}
&\fint_{Q_{r/2}^{-}(z_{0})}|D_{\beta}u-\ell_{\beta}^{r,z_{0}}|^{q}+\fint_{Q_{r/2}^{-}(z_{0})}|U-\ell_{d}^{r,z_{0}}|^{q}\\
&\leq NC_{0}^{q}r^{q\delta'}+\fint_{Q_{r/2}^{-}(z_{0})}|A^{d\beta}D_{\beta}u-\bar{A}^{d\beta}(x^{d})D_{\beta}u+\bar{g}_{d}-g_{d}|^{q}\\
&\leq NC_{0}^{q}r^{q\delta'}.
\end{align*}
Therefore, we prove the claim \eqref{difference Du U}.
It follows from \eqref{difference Du U}, \eqref{difference ell} and the continuity of $D_{\beta}u$ and $U$ in \eqref{Lip Du} that
\begin{equation}
                \label{eq2.12}
D_{\beta}u(z_0)=\ell_\beta^{z_0},\quad |D_{\beta}u(z_0)-\ell_\beta^{r,z_0}|\le NC_0 r^{\delta'},\quad\beta=1,\dots,d-1,
\end{equation}
and
\begin{align}\label{limit ld}
U(z_0)=\ell_d^{z_0},\quad |U(z_0)-\ell_d^{r,z_0}|\le NC_0 r^{\delta'}.
\end{align}

Let $z_1=(t_0-r^2,0)$. Similar to \eqref{limit l}, \eqref{eq2.12}, and \eqref{limit ld}, we have for any $r\in (0,1/4)$, under the coordinate system associated with $z_{1}$,
\begin{equation}\label{limit l00}
\begin{split}
u(z_1)&=\ell_0^{z_1},\quad |u(z_1)-\ell_0^{r,z_1}|\le NC_0 r^{1+\delta'},\\
D_{y^{\beta}}u(z_1)&=\ell_\beta^{z_1},\quad |D_{y^{\beta}}u(z_1)-\ell_\beta^{r,z_1}|\le NC_0 r^{\delta'},\quad\beta=1,\dots,d-1,
\end{split}
\end{equation}
and
\begin{align}\label{limit ld00}
\tilde U(z_1)=\ell_d^{z_1},\quad |\tilde U(z_1)-\ell_d^{r,z_1}|\le NC_0 r^{\delta'},
\end{align}
where  $D_{y^{\beta}}$ denotes the derivatives with respect to the space variables in the coordinate system associated with $z_{1}$ and $\tilde U=A^{d\beta}D_{y^{\beta}}u-g_{d}$.

{\bf Step 3. Estimate of $\langle u\rangle_{1+\delta'}$.}
Finally, we are going to estimate $\langle u\rangle_{1+\delta'}$. If $|t_{0}-t_{1}|^{1/2}=r>1/32$, then \eqref{angle u t} is obvious. If $r\leq1/32$,  then by the triangle inequality, we have
\begin{equation}\label{estimate u}
|u(z_{0})-u(z_{1})|^{q}\leq|u(z_{0})-\ell_0^{2r,z_0}|^{q}+|\ell_0^{2r,z_0}-\ell_0^{2r,z_1}|^{q}+|u(z_{1})-\ell_0^{2r,z_1}|^{q}.
\end{equation}
Next we estimate $|\ell_0^{2r,z_0}-\ell_0^{2r,z_1}|$. 
Noting that
\begin{align*}
|\ell_0^{2r,z_0}-\ell_0^{2r,z_1}|^{q}
&\leq|p^{2r,z_0}-p^{2r,z_1}|^{q}+\Big|\sum_{\beta=1}^{d-1}\ell_\beta^{2r,z_{0}}x^{\beta}
-\sum_{\beta=1}^{d-1}\ell_\beta^{2r,z_{1}} y^{\beta}\Big|^{q}\\
&\quad+\Bigg|\int_{0}^{x^{d}}\Big(\bar{A}^{dd}(z'_{0},s)\Big)^{-1}
\Big(\bar{g}_{d}(z'_{0},s)+
\ell_{d}^{2r,z_{0}}-\sum_{\beta=1}^{d-1}
\bar{A}^{d\beta}(z'_{0},s)\ell_{\beta}^{2r,z_{0}}\Big)\ ds\\
&\quad-\int_{0}^{y^{d}}\Big(\bar{A}^{dd}(z'_{1},s)\Big)^{-1}
\Big(\bar{g}_{d}(z'_{1},s)+
\ell_{d}^{2r,z_{1}}-\sum_{\beta=1}^{d-1}
\bar{A}^{d\beta}(z'_{1},s)\ell_{\beta}^{2r,z_{1}}\Big)\ ds\Bigg|^{q}.
\end{align*}
For the first term, by using the triangle inequality and \eqref{eq12.21}, we have
\begin{align}\label{est first}
\fint_{Q_{r}^{-}(z_{1})}|p^{2r,z_0}-p^{2r,z_1}|^{q}&\leq\fint_{Q_{r}^{-}(z_{1})}|u(z)-p^{2r,z_0}|^{q}+\fint_{Q_{r}^{-}(z_{1})}|u(z)-p^{2r,z_1}|^{q}\nonumber\\
&\leq\fint_{Q_{2r}^{-}(z_{0})}|u(z)-p^{2r,z_0}|^{q}+NC_{0}^{q}r^{q(1+\delta')}\leq NC_{0}^{q}r^{q(1+\delta')}.
\end{align}
In order to estimate the last two terms, we first obtain from the estimate $I-X^{-1}$ in the case 1 of the proof of Proposition  \ref{main prop} that
\begin{equation}\label{difference x tilde x}
|x-y|\leq Nr^{1+\delta'}.
\end{equation}
Then for the second term, under the coordinate system associated with $z_{0}$, it follows from the triangle inequality, \eqref{difference x tilde x}, \eqref{limit l}, \eqref{limit l00}, \eqref{Lip Du}, and \eqref{diffe coor} that for $\beta=1,\dots,d-1$,
\begin{align}\label{est second}
|\ell_\beta^{2r,z_{0}}x^{\beta}-\ell_\beta^{2r,z_{1}} y^{\beta}|
&\leq|(\ell_\beta^{2r,z_{0}}-\ell_\beta^{2r,z_{1}})x^{\beta}+\ell_\beta^{2r,z_{1}} (x^{\beta}-y^{\beta})|\nonumber\\
&\leq r\Big(|\ell_\beta^{2r,z_{0}}-D_{\beta}u(z_{0})|+|D_{\beta}u(z_{0})-D_{\beta}u(z_{1})|+| D_{\beta}u(z_{1})-D_{y^{\beta}}u(z_{1})|\nonumber\\
&\quad+|D_{y^{\beta}}u(z_{1})-\ell_\beta^{2r,z_{1}}|\Big)+Nr^{1+\delta'}\nonumber\\
&\leq NC_{0}r^{1+\delta'}.
\end{align}
For the third term, by using the triangle inequality, $\bar{A}^{dd}\geq\nu$, Lemma \ref{difference holder}, and \eqref{difference x tilde x}, we have
\begin{align}\label{est third1}
&\Bigg|\int_{0}^{x^{d}}\Big(\bar{A}^{dd}(z'_{0},s)\Big)^{-1}
\bar{g}_{d}(z'_{0},s)\ ds-\int_{0}^{y^{d}}\Big(\bar{A}^{dd}(z'_{1},s)\Big)^{-1}
\bar{g}_{d}(z'_{1},s)\ ds\Bigg|^{q}\nonumber\\
&\leq\Bigg|\int_{0}^{x^{d}}\Big(\bar{A}^{dd}(z'_{0},s)\Big)^{-1}
\Big(\bar{g}_{d}(z'_{0},s)-\bar{g}_{d}(z'_{1},s)\Big)\ ds\Bigg|^{q}\nonumber\\
&\quad+\Bigg|\int_{0}^{x^{d}}\Bigg(\Big(\bar{A}^{dd}(z'_{0},s)\Big)^{-1}-\Big(\bar{A}^{dd}(z'_{1},s)\Big)^{-1}\Bigg)\bar{g}_{d}(z'_{1},s)\ ds\Bigg|^{q}\nonumber\\
&\quad+\Bigg|\int_{y^d}^{ x^{d}}\Big(\bar{A}^{dd}(z'_{1},s)\Big)^{-1}
\bar{g}_{d}(z'_{1},s)\ ds\Bigg|^{q}\leq NC_{0}^{q}r^{q(1+\delta')}.
\end{align}
Similarly, by using the triangle inequality, $\bar{A}^{dd}\geq\nu$, Lemma \ref{difference holder}, \eqref{difference x tilde x}, \eqref{limit ld}, \eqref{Lip Du}, \eqref{diffe coor}, and \eqref{limit ld00}, we have
\begin{align}\label{est third2}
&\Bigg|\int_{0}^{x^{d}}\Big(\bar{A}^{dd}(z'_{0},s)\Big)^{-1}
\ell_{d}^{2r,z_{0}}\ ds-\int_{0}^{y^{d}}\Big(\bar{A}^{dd}(z'_{1},s)\Big)^{-1}
\ell_{d}^{2r,z_{1}}\ ds\Bigg|^{q}\nonumber\\
&\leq\Bigg|\int_{0}^{x^{d}}\Big(\bar{A}^{dd}(z'_{0},s)\Big)^{-1}
(\ell_{d}^{2r,z_{0}}-\ell_{d}^{2r,z_{1}})\ ds\nonumber\\
&\quad+\int_{0}^{x^{d}}\Bigg(\Big(\bar{A}^{dd}(z'_{0},s)\Big)^{-1}-\Big(\bar{A}^{dd}(z'_{1},s)\Big)^{-1}
\Bigg)\ell_{d}^{2r,z_{1}}\ ds\Bigg|^{q}+\Bigg|\int_{y^d}^{x^{d}}\Big(\bar{A}^{dd}(z'_{1},s)\Big)^{-1}
\ell_{d}^{2r,z_{1}}\ ds\Bigg|^{q}\nonumber\\
&\leq N|x^{d}|^q
|\ell_{d}^{2r,z_{0}}-U(z_{0})|^{q}
+N|x^{d}|^q|U(z_{0})-U(z_{1})|^{q}
+N|x^{d}|^q|U(z_{1})-\ell_{d}^{2r,z_{1}}|^{q}+Nr^{q(1+\delta')}\nonumber\\
&\leq NC_{0}^{q}r^{q(1+\delta')}+N|x^{d}|^q|U(z_{1})-\tilde U(z_{1})|^{q}+N|x^{d}|^q|\tilde U(z_{1})-\ell_{d}^{2r,z_{1}}|^{q}\nonumber\\
&\leq NC_{0}^{q}r^{q(1+\delta')}.
\end{align}
Also, we have
\begin{align}\label{est third3}
&\Bigg|-\int_{0}^{x^{d}}\Big(\bar{A}^{dd}(z'_{0},s)\Big)^{-1}
\sum_{\beta=1}^{d-1}
\bar{A}^{d\beta}(z'_{0},s)\ell_{\beta}^{2r,z_{0}}\ ds\nonumber\\
&\quad+\int_{0}^{y^{d}}\Big(\bar{A}^{dd}(z'_{1},s)\Big)^{-1}
\sum_{\beta=1}^{d-1}
\bar{A}^{d\beta}(z'_{1},s)\ell_{\beta}^{2r,z_{1}}\ ds\Bigg|^{q}\leq NC_{0}^{q}r^{q(1+\delta')}.
\end{align}

Now, coming back to \eqref{estimate u}, taking the average over $z\in Q_{r}^{-}(z_{1})$ and taking the $q$th root, using \eqref{est first}, and \eqref{est second}--\eqref{est third3}, we have
\begin{equation*}
|u(z_0)-u(z_1)|\leq NC_0 r^{1+\delta'}.
\end{equation*}
Therefore, we finish the proof of \eqref{angle u t}.
\end{proof}

\section{Weighted weak type-$(1,1)$ estimates}\label{sec thm3}

This section is devoted to the proof of a global weak type-$(1,1)$ estimate with respect to $A_{1}$ Muckenhoupt weights for solutions to
$$
\mathcal{P}u=\Div f,$$
with the coefficients $A=(A^{\alpha\beta})$ satisfying the following condition.
\begin{assumption}\label{assump omega}
(1) $A$ is of piecewise Dini mean oscillation in $\cQ$, and  there exists some constant $c_{0}>0$ such that for any $r\in(0,1/2)$, $\omega_{A}(r)\leq c_{0}(\ln r)^{-2}$.
	
(2) For some constant $c_1,c_2>0$, $\omega_0'(R_0^-)\geq c_{1}$ and for any $R\in(0,R_0/2)$, $\omega_0(R)\leq c_{2}(\ln R)^{-2}$.
\end{assumption}

We say $w: \mathbb R^{d+1}\rightarrow[0,\infty)$ belongs to $A_{1}$ if there exists some constant $C$ such that for all parabolic cylinders $Q$ in $\mathbb R^{d+1}$,
$$\fint_{Q}w(s,y)\ ds\ dy\leq C\inf_{z\in Q}w(z).$$
The $A_{1}$ constant $[w]_{A_{1}}$ of $w$ is defined as the infimum of all such $C's$.
In order to state our result, we first denote
$$
w(\cQ):=\int_{\cQ}w(z)\ dz,\quad  \|f\|_{L_{p,w}(\cQ)}:=\Big(\int_{\cQ}|f|^{p}w\ dz\Big)^{1/p},~ p\in[1,\infty),
$$
and then introduce the weighted Sobolev space:
$$\mathcal{H}_{p,w}^{1}(\cQ):=\{u: u_{t}\in \mathbb{H}_{p,w}^{-1}(\cQ), u,Du\in L_{p,w}(\cQ)\},$$
where
$$\mathbb{H}_{p,w}^{-1}(\cQ):=\Big\{f: f=\sum_{|\alpha|\leq1}D^{\alpha}f_{\alpha}, f_{\alpha}\in L_{p,w}(\cQ)\Big\},$$
$$\|f\|_{\mathbb{H}_{p,w}^{-1}(\cQ)}:=\inf\Big\{\sum_{|\alpha|\leq1}\|f_{\alpha}\|_{L_{p,w}(\cQ)}: f=\sum_{|\alpha|\leq1}D^{\alpha}f_{\alpha}\Big\},$$
and
$$\|u\|_{\mathcal{H}_{p,w}^{1}(\cQ)}:=\|u_{t}\|_{\mathbb{H}_{p,w}^{-1}(\cQ)}+\sum_{|\alpha|\leq1}\|D^{\alpha}u\|_{L_{p,w}(\cQ)}.$$
Denote $\delta_{0}:=\min_{1\leq j\leq M-1}\mbox{dist}\{\partial_{p} \cQ_{j}, (-T,0)\times \partial\cD\}$.

Recall that in the proof of Lemma \ref{weak est barv}, Lemmas \ref{lemma weak} and \ref{solvability} are the key points. We will use generalizations of the two lemmas since our estimate and argument depend on the coordinate system, one of them is stated below. Let $\{Q_{\alpha}^{k}\}$ be a collection of dyadic ``parabolic cubes'' in $\cQ$. See \cite[Theorem~11]{Ch90} and also the proof of \cite[Lemma 4.1]{dek}.
Let $p,c\in (1,\infty)$.
\begin{assumption}
	\label{assump9.26}
	i) $S$ is a bounded linear operator on $L_{p,w}(\cQ)$.
	
	ii) If for some $f\in L_{p,w}(\cQ)$, $s>0$, and some cube $Q_\alpha^k$ we have
	$$
	s<\frac{1}{w(Q_\alpha^k)}\int_{Q_\alpha^k}|f|w\ dz\leq C_{0}s,
	$$
	then $f$ admits a decomposition $f=g+b$ in $Q_\alpha^k$, where $g$ and $b$ satisfy
	\begin{equation*}
	\int_{Q_\alpha^k}|g|^{p}w\ dz\leq C_1s^{p}w(Q_\alpha^k),\quad\int_{\cQ\setminus Q_{cr}(z_{0})}\big|S(b\chi_{Q_\alpha^k})\big|w\ dz\leq C_1sw(Q_\alpha^k)
	\end{equation*}
	with $z_0\in Q_\alpha^k$ and $r=\text{diam}\,Q_\alpha^k$.
\end{assumption}
Then the same proof as in \cite[Lemma 6.3]{dx} gives the following result.
\begin{lemma}\label{weak general}
	Under Assumption \ref{assump9.26}, for any $f\in L_{p,w}(\cQ)$ and $s>0$, we have
	\begin{equation*}
	w(\{z\in \cQ: |Sf(z)|>s\})\leq\frac{N}{s}\int_{\cQ}|f|w\ dz,
	\end{equation*}
	where $N=N(d,c,\cQ,C_1,\|S\|_{L_{p,w}\rightarrow L_{p,w}})$ is a constant. Moreover, $S$ can be extended to a bounded operator from $L_{1,w}(\cQ)$ to weak-$L_{1,w}(\cQ)$.
\end{lemma}

The generalization of Lemma \ref{solvability} and its proof will be given in the Appendix. Now we state our global weak type-$(1,1)$ estimate with $A_{1}$ weights.

\begin{theorem}\label{thm3}
Let $p\in(1,\infty)$, $\cQ:=(-T,0)\times\cD$, $\cD$ have a $C^{1,\text{Dini}}$ boundary,  $\cQ_{1},\dots,\cQ_{M-1}$ be away from $(-T,0)\times \partial\cD$ and satisfy the conditions in Theorem \ref{thm1}. Let $w$ be an $A_{1}$ Muckenhoupt weight and Assumption \ref{assump omega} be satisfied. For $f\in L_{p,w}((-T,0)\times\cD)$ with $T\in(0,\infty)$, let $u\in \mathcal{H}_{p,w}^{1}(\cQ)$ be a weak solution to
\begin{align*}
	\begin{cases}
	\mathcal{P}u=\Div f&\ \mbox{in}~\cQ,\\
	u=0&\ \mbox{on}~\partial_p \cQ.
	\end{cases}
\end{align*}
Then for any $s>0$, we have
$$w\Big(\{(t,x)\in \cQ: |Du(t,x)|>s\}\Big)\leq\frac{N}{s}\|f\|_{L_{1,w}(\cQ)},$$
where $N$ depends on $n,d,M,\omega_{A},\nu,\Lambda,p,\delta_{0},[w]_{A_1}, T$, the $C^{1,\text{Dini}}$ characteristics of $\cD$, $\cQ_{j}$, and $C^{\gamma_{0}}$ norms of $\cQ_{j}$ with respect to $x$ and $t$, respectively. Moreover, the linear operator $S: f\mapsto Du$ can be extended to a bounded operator from $L_{1,w}(\cQ)$ to weak-$L_{1,w}(\cQ)$.
\end{theorem}

\begin{proof}
We follow a similar argument in the proof of \cite[Theorem 5.2]{dx}, in which the weighted $W_{w}^{1,p}$-solvability and estimates for divergence form elliptic systems are the important ingredients. Here we use the $\mathcal{H}_{p,w}^{1}$-solvability and estimates with $A_{p}$ weights, Lemma \ref{sol weight}, for divergence form parabolic systems, to conclude that the map $S: f\mapsto Du$ is a bounded linear map on $L_{p,w}(\cQ)$. We need to show that $S$ verifies the conditions of Lemma \ref{weak general}.

For a fixed $z_{k}=(t_{k},x_{k})\in Q_{\alpha}^{k}$, we associate $Q_{\alpha}^{k}$ with a parabolic cylinder $Q_k=Q_{r_k}^{-}(z_{k})$ such that $z_{k}\in Q_{\alpha}^{k}\subset Q_{k}$, where $r_k=\text{diam}\, Q_{\alpha}^{k}\leq {\delta_{0}}/{2}$. Suppose for some $Q_{\alpha}^{k}$ and $s>0$,
	\begin{equation}\label{prop f Ql}
	s<\frac{1}{w(Q_{\alpha}^{k})}\int_{Q_{\alpha}^{k}}|f|w\ dz\leq C_{0}s.
	\end{equation}
We need to check that $f$ enjoys the Assumption \ref{assump9.26}, i.e., $f$ admits a decomposition in $Q_{\alpha}^{k}$.
	
	(i) If $\dist(x_{k},\partial\cD)\leq {\delta_{0}}/{2}$, then $Q_k$ does not intersect with subdomains $\cQ_{j}$, $j=1,\dots,M-1$. In this case, we choose the coordinate system according to $\tilde{z}_{k}\in(-T,0)\times\partial\cD$, which satisfies $|z_{k}-\tilde{z}_{k}|_{p}=\dist(x_{k},\partial\cD)$. Let
	$$
	g:=\fint_{Q_{\alpha}^{k}}f\ dz,\quad b=f-g\quad\mbox{in}~Q_{\alpha}^{k}.
	$$
	Then
	\begin{equation*}
	\fint_{Q_{\alpha}^{k}}b\ dz=0
	\end{equation*}
	and
	\begin{align*}
	|g|\leq \fint_{Q_{\alpha}^{k}}|f|\ dz\leq \frac{1}{|Q_{\alpha}^{k}|\inf\limits_{Q_{\alpha}^{k}}w}\int_{Q_{\alpha}^{k}}|f|w\ dz\leq\frac{N}{w(Q_{\alpha}^{k})}\int_{Q_{\alpha}^{k}}|f|w\ dz\leq NC_{0}s,
	\end{align*}
	where we used the definition of $w$ and \eqref{prop f Ql}.
	Hence,
	$$\int_{Q_{\alpha}^{k}}|g|^{p}w\ dz\leq NC^p_{0}s^{p}w(Q_{\alpha}^{k}).$$
	
	Let $u_{1}\in \mathcal{H}_{p,w}^{1}(\cQ)$ be the unique weak solution of
	\begin{align*}
	\begin{cases}
	\mathcal{P}u_{1}=\Div b&\ \mbox{in}~\cQ,\\
	u_{1}=0&\ \mbox{on}~\partial_p \cQ.
	\end{cases}
	\end{align*}
Let $p'=p/(p-1)$ and $\mathcal{P}^{*}$ be the adjoint operator of $\mathcal{P}$.
	Set $c=\frac{4R_0}{\delta_{0}}$ with $R_{0}=\mbox{diam}~\cQ$. Then for any $R\geq cr_{k}$ such that $\cQ\setminus Q_{R}(z_{k})\neq\emptyset$ and $h\in C_{0}^{\infty}(\cQ_{2R}(z_{k})\setminus Q_{R}(z_{k}))$, let $u_{2}\in \mathcal{H}_{p',w^{-\frac{1}{p-1}}}^{1}(\cQ)$ be a weak solution of
	\begin{align*}
	\begin{cases}
	\mathcal{P}^{*}u_{2}=\Div h&\ \mbox{in}~\cQ,\\
	u_{2}=0&\ \mbox{on}~((-T,0)\times\partial \cD)\cup (\{t=0\}\times\overline{\cD}),
	\end{cases}
	\end{align*}
	which satisfies
	\begin{equation*}
	\left(\int_{\cQ}|Du_{2}|^{p'}w^{-\frac{1}{p-1}}\ dz\right)^{\frac{1}{p'}}\leq N\left(\int_{\cQ}|h|^{p'}w^{-\frac{1}{p-1}}\ dz\right)^{\frac{1}{p'}}=N\left(\int_{\cQ_{2R}(x_{k})\setminus Q_{R}(z_{k})}|h|^{p'}w^{-\frac{1}{p-1}}\ dz\right)^{\frac{1}{p'}}.
	\end{equation*}
	See Lemma \ref{sol weight}.
	Then by using the definition of adjoint solutions, the fact that $b$ is supported in $Q_{\alpha}^{k}$ with mean zero, and $h\in C_{0}^{\infty}(\cQ_{2R}(z_{k})\setminus Q_{R}(z_{k}))$, we obtain
	\begin{align}\label{esti Du h}
	\int_{\cQ_{2R}(z_{k})\setminus Q_{R}(z_{k})}Du_{1}\cdot h
	=\int_{Q_{\alpha}^{k}}Du_{2}\cdot b=\int_{Q_{\alpha}^{k}}\big(Du_{2}-Du_{2}(z_{k})\big)\cdot b.
	\end{align}
	Since $R\leq R_0$, $Q_{\frac{\delta_{0}R}{2R_0}}(z_{k})$ does not intersect with subdomains $\cQ_{j}$, $j=1,\dots,M-1$. Because $\mathcal{P}^{*}u_{2}=0$ in $\cQ_{R}(z_{k})$, by flattening the boundary and using a similar argument that led to an a priori estimate of the modulus of continuity of $Du_{2}$ in the proof of \cite[(4.22)]{dek2}, we have
\begin{align*}
	|Du_{2}(z)-Du_{2}(z_{k})|\leq N\left(\Big(\frac{|z-z_{k}|_{p}}{R}\Big)^{\gamma}
		+\int_{0}^{2|z-z_{k}|_{p}}\frac{\hat\omega_{A}(s)}{s}\ ds\right)
		R^{-d-2}\|Du_{2}\|_{L_{1}(\cQ_{\frac{\delta_{0}R}{2R_0}}(z_{k}))}
\end{align*}
for any $z\in Q_{\alpha}^{k}\subset\cQ_{\frac{\delta_{0}R}{4R_0}}(z_{k})$, where $\gamma\in(0,1)$ is a constant and $\hat\omega_{A}(s)$ is defined as in \cite[(4.15)]{dek2}, which is derived from $\omega_{A}(s)$. Then, coming back to \eqref{esti Du h}, using a similar argument in the proof of \cite[Theorem 5.2]{dx} and Lemma \ref{weak est barv}, we obtain
\begin{equation}
                                \label{eq5.13}
\int_{\cQ\setminus Q_{cr_{k}}(z_{k})}|Du_1|w\ dz\leq Nsw(Q_{\alpha}^{k}).
\end{equation}
	
(ii) If $\dist(x_{k},\partial\cD)\geq {\delta_{0}}/{2}$, then $Q_k$ does not intersect with $(-T,0)\times \partial\cD$. In this case, we choose the coordinate system according to $z_{k}$. The rest proof is the same as that in \cite[Theorem 5.2]{dx} and we also obtain \eqref{eq5.13}.
Therefore, $S$ satisfies the hypothesis of Lemma \ref{weak general}, and thus for any $s>0$,
$$w(\{z\in \cQ: |Du(z)|>s\})\leq\frac{N}{s}\|f\|_{L_{1,w}(\cQ)}.$$
The theorem is proved.
\end{proof}

\section{Application: Regularity for parabolic transmission problems}\label{transmission prb}

Caffarelli, Soria-Carro, and  Stinga \cite{css} recently proved existence, uniqueness, and optimal regularity of solutions to transmission problems for harmonic functions with $C^{1,\alpha}$ interfaces.  Their argument is mainly based on the mean value property and the maximum principle for harmonic functions and an approximation argument. In \cite{dong}, an alternative proof of the result in \cite{css} is given, which works for more general elliptic systems with multiple subdomains and $C^{1,\text{Dini}}$ interfaces. The main idea of the proof in \cite{dong} is to reduce the transmission problems to elliptic systems with piecewise H\"{o}lder or Dini continuous non-homogeneous terms by solving a Laplace equation with conormal boundary data. Then the results follows by these in \cite{d,dx}. In this section, we extend the results in \cite{css, dong} to parabolic systems by using Theorems \ref{thm1} and \ref{thm holder}.

\subsection{Main results for the transmission problem}
We first  introduce some notation. For $\cQ:=(-T,0)\times\cD$, we denote
$$B\cQ:=\{t=-T\}\times\overline{\cD},\quad S\cQ:=(-T,0)\times\partial\cD,\quad\partial_{p}\cQ:=B\cQ\cup S\cQ.$$
In the following, we assume that $\cQ_{j}=(-T,0)\times \cD_j$ are cylindrical, $j=1,\ldots,M$, and similarly define $B\cQ_{j}$ and $S\cQ_{j}$. Without loss of generality, we assume that $\cD_j\subset\subset \cD$ for $j=1,\ldots,M-1$ and  $\partial\cD\subset \partial\cD_M$. For $\gamma\in(0,1]$, we denote $C^{(1+\gamma)/2,1+\gamma}(\cQ)$ to be the space of functions with finite norm $|u|_{(1+\gamma)/2,1+\gamma;\cQ}$.
The transmission problem is given by
\begin{align}\label{trans prb}
\begin{cases}
\displaystyle \mathcal{P}_{0}u:=-u_{t}+D_{\alpha}(A^{\alpha\beta}D_{\beta}u)=\Div F+f& \mbox{in}~\cup_{j=1}^{M}\cQ_{j},\\
u=0& \mbox{on}~\partial_{p}\cQ,\\
u\big|_{S\cQ_{j}}^{+}=u\big|_{S\cQ_{j}}^{-},\quad A^{\alpha\beta}D_{\beta}u\nu_{\alpha}\big|_{S\cQ_{j}}^{+}
-A^{\alpha\beta}D_{\beta}u\nu_{\alpha}\big|_{S\cQ_{j}}^{-}=g_{j},& j=1,\dots,M-1,
\end{cases}
\end{align}
where $\nu_{\alpha}$ is the unit normal vector on $S\cQ_{j}$ pointing inside $S\cQ_{j}$, $u\big|_{S\cQ_{j}}^{+}$ and $u\big|_{S\cQ_{j}}^{-}$ ($A^{\alpha\beta}D_{\beta}u\nu_{\alpha}\big|_{S\cQ_{j}}^{+}$ and $A^{\alpha\beta}D_{\beta}u\nu_{\alpha}\big|_{S\cQ_{j}}^{-}$) are the left and right limits of $u$ (its conormal derivatives) on $S\cQ_{j}$, respectively, $j=1,\dots,M-1$.

The first result of this section is about the case when the interfaces are $C^{1,\mu}$ in the spatial variables and the coefficients and data are piecewise H\"{o}lder continuous.
\begin{theorem}\label{trans holder}
Assume that $\partial\cD_j$ are $C^{1,\mu}$, $A^{\alpha\beta}$ and $F$ are piecewise  $C^{\delta/2,\delta}$ with $\delta\in \Big(0,\mu/(1+\mu)\Big]$, and $g_{j}\in C^{\delta/2,\delta}(S\cQ_{j})$, $j=1,\dots,M-1$. Then there exists a unique weak solution $u\in \mathcal{H}_{2}^{1}(\cQ)$ to \eqref{trans prb}, which is piecewise $C^{(1+\delta)/2,1+\delta}$ up to $S\cQ_{j}$, $j=1,\dots,M$, and satisfies
$$\sum_{j=1}^{M}|u|_{(1+\delta)/2,1+\delta;\overline{\cQ_{j}}}\leq N\bigg(\sum_{j=1}^{M-1}|g_{j}|_{\delta/2,\delta;S\cQ_{j}}+\sum_{j=1}^{M}|F|_{\delta/2,\delta;\overline{\cQ}_{j}}
+\|f\|_{L_{\infty}(\cQ)}\bigg),$$
where $N$ depends on $n$, $d$, $M$, $\delta$, $\mu$, $\nu$, $\Lambda$, $\cQ_{j}$, and  $\|A\|_{C^{\delta/2,\delta}(\overline{\cQ}_{j})}$.
\end{theorem}

\begin{remark}
In the special case when $M=2$ or when $A^{\alpha\beta}$ and $F$ are H\"older continuous in the whole domain, by using \cite{d} and the linearity the result of Theorem \ref{trans prb} still holds with $\delta=\mu$.
\end{remark}

Our second result is concerned with the case when the interfaces are  $C^{1,\text{Dini}}$ in the spatial variables and the coefficients and data are of piecewise Dini mean oscillation. 

\begin{theorem}\label{trans dini}
Assume that $\partial\cD_{j}$ are $C^{1,\text{Dini}}$, $A^{\alpha\beta}$ and $F$ are of piecewise Dini mean oscillation in $\cQ$, $F, f\in L_{\infty}(\cQ)$, and $g_{j}$ are Dini continuous on $S\cQ_{j}$, $j=1,\dots,M-1$. Then there exists a unique weak solution $u\in \mathcal{H}_{2}^{1}(\cQ)$ to \eqref{trans prb}, which is piecewise $C^{1/2,1}$ up to $S\cQ_{j}$, $j=1,\dots,M$.
\end{theorem}

\subsection{Proofs of Theorems \ref{trans holder} and \ref{trans dini}}
The proof is a modification of \cite{dong}.

\begin{proof}[Proof of Theorem \ref{trans holder}]
For a fixed $j=1,\dots,M-1$ and a point $x_{jk}\in\partial\cD_{j}$, there is a neighbourhood $V_{jk}$ of $x_{jk}$ and  a $C^{1,\mu}$ diffeomorphism ${\bf\Phi}_{jk}$ from $V_{jk}$ onto a unit ball $B:=B_{1}(0)\subset\mathbb R^{d}$ such that $${\bf\Phi}_{jk}(V_{jk}\cap\partial\cD_{j})\subset\partial\mathbb R^{d}_{+},\quad \det D{{\bf\Phi}_{jk}}=1.$$
Let
$$y={\bf\Phi}_{jk}(x)=(\Phi_{jk}^{1}(x),\dots,\Phi_{jk}^{d}(x)),\quad x=({{\bf\Phi}_{jk}})^{-1}(y)=:{\bf\Psi}_{jk}(y)\quad  (\det D{{\bf\Psi}_{jk}}=1),$$
and
$$g:=g(t,y')=g_{j}(t,x').$$
Since $\partial\cD_{j}$ is $C^{1,\mu}$, there exist finitely many points $x_{jk}\in\partial\cD_{j}$ and $V_{jk}\subset\cD$, $k=1,\dots,m$, such that $\partial\cD_{j}\subset\cup_{k=1}^{m}V_{jk}$. Let $\{\zeta_{jk}\}_{k=1}^{m}$ be a smooth partition of unity subordinate to $V_{jk}$. Denote $B^{+}:=B\cap\{y^{d}>0\}$, $Q^{+}:=(-T,0)\times B^{+}$, $\Gamma:=B\cap\{y^{d}=0\}$, and $SQ^{+}:=(-T,0)\times \Gamma$. Let  $v_{jk}\in\mathcal{H}_{2}^{1}(Q^{+})$ be the weak solution to
\begin{align*}
\begin{cases}
-\partial_{t}v_{jk}+\Delta v_{jk}=0&\mbox{in}~Q^{+},\\
\partial_{\nu}v_{jk}\big|_{SQ^{+}}^{+}=\frac 1 2\zeta_{jk}\circ{\bf\Psi}_{jk}g,\\
v_{jk}=0&\mbox{on}~\partial_{p}Q^{+}\setminus SQ^{+}.
\end{cases}
\end{align*}
The existence and uniqueness follows from \cite[Theorem 6.46]{lgm}. We take the even extension $\bar{v}_{jk}$ of $v_{jk}$ with respect to $\{y^{d}=0\}$ defined by
\begin{align*}
\bar{v}_{jk}(t,y)=
\begin{cases}
v_{jk}(t,y)& \mbox{in}~(-T,0)\times B^{+},\\
v_{jk}(t,y',-y^{d})& \mbox{in}~(-T,0)\times B^{-},
\end{cases}
\end{align*}
where $B^{-}:=B\cap\{y^{d}<0\}$. Then $\bar{v}_{jk}$ satisfis
\begin{align}\label{eq bar v}
\begin{cases}
-\partial_{t}\bar{v}_{jk}+\Delta \bar{v}_{jk}=0& \mbox{in}~(-T,0)\times B,\\
\partial_{\nu}\bar{v}_{jk}\big|_{SQ^{+}}^{+}-\partial_{\nu}\bar{v}_{jk}\big|_{SQ^{+}}^{-}
=\zeta_{jk}\circ{\bf\Psi}_{jk}g,\\
\bar{v}_{jk}\big|_{SQ^{+}}^{+}=\bar{v}_{jk}\big|_{SQ^{+}}^{-},\\
\bar{v}_{jk}=0& \mbox{on}~\partial_{p}((-T,0)\times B).
\end{cases}
\end{align}

Next we transform back to the $x$-variables. Let $\tilde v_{jk}(t,x)=\bar v_{jk}(t,{{\bf\Phi}_{jk}}(x))$ and
$$
a^{\alpha\beta}:=a^{\alpha\beta}(x)
=D_{l}{\bf\Psi}_{jk}^{\alpha}(y)D_{l}{\bf\Psi}_{jk}^{\beta}(y).
$$
Choose a cut-off function $\eta_{jk}:=\eta_{jk}(x)\in C_{0}^{\infty}(V_{jk})$ satisfying $0\leq\eta_{jk}\leq1$ and $\eta_{jk}\equiv1$ on the support of $\zeta_{jk}$.
From \eqref{eq bar v}, we obtain
\begin{align}\label{eq tildevk}
\begin{cases}
-\partial_{t}\tilde{v}_{jk}+D_{\alpha}(a^{\alpha\beta}D_{\beta}\tilde{v}_{jk})=0
& \mbox{in}~(-T,0)\times V_{jk},\\
a^{\alpha\beta}D_{\beta}\tilde{v}_{jk}\nu_{\alpha}\big|_{(-T,0)\times(V_{jk}\cap\partial\cD_{j})}^{+}-a^{\alpha\beta}D_{\beta}\tilde{v}_{jk}\nu_{\alpha}\big|_{(-T,0)\times(V_{jk}\cap\partial\cD_{j})}^{-}=\zeta_{jk}g_{j},\\
\tilde{v}_{jk}\big|_{(-T,0)\times(V_{jk}\cap\partial\cD_{j})}^{+}=\tilde{v}_{jk}\big|_{(-T,0)\times(V_{jk}\cap\partial\cD_{j})}^{-},\\
\tilde{v}_{jk}=0& \mbox{on}~\partial_{p}((-T,0)\times V_{jk}),
\end{cases}
\end{align}
which together with a direct calculation yields that
\begin{align}\label{eq tilde vii}
\begin{cases}
-\partial_{t}(\tilde{v}_{jk}\eta_{jk})
+D_{\alpha}(a^{\alpha\beta}D_{\beta}(\tilde{v}_{jk}\eta_{jk}))
=D_{\alpha}(\tilde{v}_{jk}a^{\alpha\beta}D_{\beta}\eta_{jk})+a^{\alpha\beta}
D_{\alpha}\eta_{jk}D_{\beta}\tilde v_{jk}&\mbox{in}~\cQ,\\
a^{\alpha\beta}D_{\beta}(\tilde{v}_{jk}\eta_{jk})\nu_{\alpha}\big|_{S\cQ_{j}}^{+}
-a^{\alpha\beta}D_{\beta}(\tilde{v}_{jk}\eta_{jk})\nu_{\alpha}\big|_{
S\cQ_{j}}^{-}=\zeta_{jk}g_{j},\\
\tilde{v}_{jk}\eta_{jk}\big|_{(-T,0)\times(V_{jk}\cap\partial\cD_{j})}^{+}=\tilde{v}_{jk}\eta_{jk}\big|_{(-T,0)\times(V_{jk}\cap\partial\cD_{j})}^{-},\\
\tilde{v}_{jk}\eta_{jk}=0&\mbox{on}~\partial_p\cQ.
\end{cases}
\end{align}
Applying the trace lemma, $\mathcal{H}_{2}^{1}$-estimate, and \cite[Theorem 1.1]{lieberman} to \eqref{eq tildevk}, we have
\begin{equation}\label{vk H21}
\|\tilde v_{jk}\eta_{jk}\|_{\mathcal{H}_{2}^{1}((-T,0)\times V_{jk})}\leq N\|g_{j}\|_{L_{2}(S\cQ_{j})},
\end{equation}
and
\begin{equation}\label{tildev holder}
|\tilde{v}_{jk}\eta_{jk}|_{(1+\delta)/2,1+\delta;(-T,0)\times V_{jk}}\leq N|g_{j}|_{\delta/2,\delta;S\cQ_{j}}.
\end{equation}
Denote
$$v_{j}:=\sum_{k=1}^{m}\tilde v_{jk}\eta_{jk}\quad \text{in}\ \cQ.$$
Then we have from \eqref{eq tilde vii} and the identity $\sum_{k=1}^{m}\zeta_{jk}=1$ on $\partial \cD_j$ that
\begin{align}\label{eq vj}
\begin{cases}
\displaystyle -\partial_{t} v_{j}+D_{\alpha}(a^{\alpha\beta}D_{\beta} v_{j})=\sum_{k=1}^{m}D_{\alpha}(\tilde{v}_{jk}a^{\alpha\beta}
D_{\beta}\eta_{jk})+\sum_{k=1}^{m}a^{\alpha\beta}D_{\alpha}\eta_{jk}D_{\beta}\tilde v_{jk}&\mbox{in}~\cQ,\\
a^{\alpha\beta}D_{\beta}v_{j}\nu_{\alpha}\big|_{S\cQ_{j}}^{+}-a^{\alpha\beta}D_{\beta}v_{j}\nu_{\alpha}\big|_{S\cQ_{j}}^{-}=g_{j},\\
v_{j}\big|_{S\cQ_{j}}^{+}=v_{j}\big|_{S\cQ_{j}}^{-},\\
v_{j}=0&\mbox{on}~\partial_p\cQ.
\end{cases}
\end{align}
Furthermore, we have from \eqref{vk H21} and \eqref{tildev holder} that
\begin{equation}\label{hatv H21}
\|v_{j}\|_{\mathcal{H}_{2}^{1}(\cQ)}\leq N\|g_{j}\|_{L_{2}(S\cQ_{j})},
\end{equation}
and
\begin{equation}\label{vj holder}
|v_{j}|_{(1+\delta)/2,1+\delta;\cQ}\leq N|g_{j}|_{\delta/2,\delta;S\cQ_{j}}.
\end{equation}

Denote
\begin{equation}\label{def w u vj}
w:=u-\sum_{j=1}^{M-1}v_{j},
\end{equation}
then it follows from the definition of weak solutions and \eqref{eq vj} that $w$ satisfies
\begin{align}\label{trans prb22}
\begin{cases}
\mathcal{P}_{0}w=\Div \tilde F+\tilde f&\mbox{in}~\cQ,\\
w=0&\mbox{on}~\partial_{p}\cQ,
\end{cases}
\end{align}
where
$$\tilde F=1_{\cup_{j=1}^{M}\cQ_{j}}F-\sum_{j=1}^{M-1}
(A-a)Dv_{j}-\sum_{j=1}^{M-1}\sum_{k=1}^{m}\tilde{v}_{jk}aD\eta_{jk},\quad a=(a^{\alpha\beta}),$$
and
$$\tilde f=f-\sum_{j=1}^{M-1}\sum_{k=1}^{m}a^{\alpha\beta}D_{\alpha}\eta_{jk}D_{\beta}\tilde v_{jk}.$$
By using the Galerkin method, we find that there is a unique solution $w\in\mathcal{H}_{2}^{1}(\cQ)$ to \eqref{trans prb22} and thus the existence and uniqueness of $u$ is proved. Moreover,
\begin{align}\label{est u H21}
\|w\|_{\mathcal{H}_{2}^{1}(\cQ)}&\leq N\Big(\|\tilde F\|_{L_{2}(\cQ)}+\|\tilde f\|_{L_{2}(\cQ)}\Big)\nonumber\\
&\leq N\Big(\|F\|_{L_{2}(\cQ)}+\sum_{j=1}^{M-1}
\|(A^{\alpha\beta}-a^{\alpha\beta})D_{\beta}v_{j}\|_{L_{2}(\cQ)}
+\sum_{j=1}^{M-1}\sum_{k=1}^{m}\|\tilde{v}_{jk}a^{\alpha\beta}D_{\beta}\eta_{jk}\|_{L_{2}(\cQ)}\nonumber\\
&\quad+\|f\|_{L_{2}(\cQ)}+\sum_{j=1}^{M-1}\sum_{k=1}^{m}\|a^{\alpha\beta}D_{\alpha}\eta_{jk}D_{\beta}\tilde v_{jk}\|_{L_{2}(\cQ)}\nonumber\\
&\leq N\Big(\|F\|_{L_{2}(\cQ)}+\sum_{j=1}^{M-1}\|g_{j}\|_{L_{2}(S\cQ_{j})}+\|f\|_{L_{2}(\cQ)}\Big),
\end{align}
where we used \eqref{hatv H21}, the boundedness of $A^{\alpha\beta}$ and $a^{\alpha\beta}$, and \eqref{vk H21}  in the third inequality. Recalling that $\partial\cD_{j}$ are $C^{1,\mu}$, $A^{\alpha\beta}$ and $\tilde F$ are piecewise  $C^{\delta/2,\delta}$, by using Theorem \ref{thm holder}, \eqref{tildev holder}, \eqref{vj holder}, and \eqref{est u H21}, we obtain
\begin{align}\label{w holder}
&\sum_{j=1}^{M}|w|_{(1+\delta)/2,1+\delta;\overline{\cQ_{j}}}\nonumber\\
&\leq N\bigg(\sum_{j=1}^{M-1}|(A^{\alpha\beta}-a^{\alpha\beta})
D_{\beta}v_{j}|_{\delta/2,\delta;\overline{\cQ}}
+\sum_{j=1}^{M-1}\sum_{k=1}^{m}\|\tilde{v}_{jk}
a^{\alpha\beta}D_{\beta}\eta_{jk}\|_{\delta/2,\delta;\overline{\cQ}}
+\sum_{j=1}^{M}|F|_{\delta/2,\delta;\overline{\cQ}_{j}}\nonumber\\
&\quad+\|f\|_{L_{\infty}(\cQ)}+\sum_{j=1}^{M-1}\sum_{k=1}^{m}
\|a^{\alpha\beta}D_{\alpha}\eta_{jk}D_{\beta}\tilde v_{jk}\|_{L_{\infty}(\cQ)}+\|w\|_{L_{2}(\cQ)}+\|Dw\|_{L_{1}(\cQ)}\bigg)\nonumber\\
&\leq N\bigg(\sum_{j=1}^{M-1}|g_{j}|_{\delta/2,\delta;S\cQ_{j}}+\sum_{j=1}^{M}|F|_{\delta/2,\delta;\overline{\cQ}_{j}}
+\|f\|_{L_{\infty}(\cQ)}\bigg).
\end{align}
Combining \eqref{def w u vj}, the triangle inequality, \eqref{vj holder}, and \eqref{w holder}, we have
\begin{align*}
\sum_{j=1}^{M}|u|_{(1+\delta)/2,1+\delta;\overline{\cQ_{j}}}
&\leq\sum_{j=1}^{M}|w|_{(1+\delta)/2,1+\delta;\overline{\cQ_{j}}}
+\sum_{j=1}^{M-1}|v_{j}|_{(1+\delta)/2,1+\delta;\overline{\cQ}}\\
&\leq N\bigg(\sum_{j=1}^{M-1}|g_{j}|_{\delta/2,\delta;S\cQ_{j}}
+\sum_{j=1}^{M}|F|_{\delta/2,\delta;\overline{\cQ}_{j}}
+\|f\|_{L_{\infty}(\cQ)}\bigg).
\end{align*}
The theorem is proved.
\end{proof}

We say that a function $f$ is of $L_2$-Dini mean oscillation in $\cQ$ if
\begin{align*}
{\check\omega}_{A}(r):=\sup_{z_{0}\in \cQ}
\bigg(\fint_{Q_{r}^{-}(z_{0})\cap \cQ}|f(t,x)-(f)_{Q_{r}^{-}(z_{0})\cap \cQ}|^2\ dz\bigg)^{1/2}
\end{align*}
satisfies the Dini condition.

\begin{proof}[Proof of Theorem \ref{trans dini}]
If we can prove the claim: $\nabla v_{j}$ satisfies the $L_{2}$-Dini mean oscillation condition in $\cQ$, then the rest of the proof follows from that of Theorem \ref{trans holder} by using Theorem \ref{thm1}. Hence, it suffices to prove the claim. The proof is a modification of \cite[Theorem 1.7]{dlk} and \cite[Theorem 1.4]{dong}, and we only need to prove the boundary estimate since the interior estimate is simpler.

For any $z_0\in\partial\mathbb R^{d+1}=\{x^{d}=0\}$,
$$
Q_{r,+}^{-}(z_0):=Q_{r}^{-}(z_0)\cap\{x^{d}>0\} \ \ \text{and} \ \
\Gamma_{r}(z_0):=Q_{r}^{-}(z_0)\cap\{x^{d}=0\}.
$$
Recalling that $SQ_{j}$ is $C^{1,\text{Dini}}$, and as in the proof of Theorem \ref{trans holder}, we only need to prove that for any weak $u$ solution to
\begin{align}
\begin{cases}
-u_{t}+\Delta u=D_{d}g^{d}&\quad \mbox{in}~Q_{4,+}^{-}(0),\\
D_du=g^{d}&\quad \mbox{on}~\Gamma_{4}(0),
\end{cases}
\end{align}
$\nabla u$ satisfies the $L_{2}$-Dini mean oscillation condition in $Q_{1,+}^{-}(0)$,
provided that $g^{d}=g^{d}(t,x)=g^d(t,x')$ is Dini continuous satisfying $D_{d}g^{d}=0$.

For $z\in Q_{4,+}^{-}(0)$ and $r>0$, we define
$$\phi(z,r):=\left(\fint_{Q_{r}^{-}(z)\cap Q_{4,+}^{-}(0)}|Du-(Du)_{Q_{r}^{-}(z)\cap Q_{4,+}^{-}(0)}|^{2}\right)^{1/2}.$$
As in \cite{dek2}, we denote $B_{r}^{+}(x_{0})=B_{r}(x_{0})\cap\{x^{d}>0\}$ and fix a smooth set $\cD\subset\mathbb R^{d}$ satisfying $B_{2/3}^{+}(0)\subset\cD\subset B_{3/4}^{+}(0)$ and for $z_{0}=(t_{0},x_{0})\in\partial\mathbb R^{d+1}$, we denote
$$\cD_{r}(x_{0}):=r\cD+x_{0}.$$
Now we decompose $u=w+v$, where $w\in\mathcal{H}_{2}^{1}$ is the weak solution of
\begin{align*}
\begin{cases}
-w_{t}+\Delta w
=D_{d}(g^{d}-\bar{g}^{d})& \mbox{in}~(t_{0}-4r^2,t_{0})\times\cD_{2r}(x_{0}),\\
D_{\alpha}w\nu_\alpha=
(g^{d}-\bar{g}^{d})\nu_{d}& \mbox{on}~  (t_{0}-4r^2,t_{0})\times\partial\cD_{2r}(x_{0}),\\
w=0& \mbox{on}~ \{t=t_{0}-4r^2\}\times\cD_{2r}(x_{0}),
\end{cases}
\end{align*}
where
\begin{align*}
\bar{g}^{d}=\fint_{Q_{2r,+}^{-}(z_0)}g^{d}(t,x)\ dx \ dt.
\end{align*}
By using the $\mathcal{H}_{2}^{1}$-estimate and $ Q_{r,+}^{-}(z_{0})\subset(t_{0}-4r^2,t_{0})\times\cD_{2r}(x_{0})\subset Q_{2r,+}^{-}(z_{0})$, we have
\begin{align}\label{Dw est}
\left(\fint_{Q_{r,+}^{-}(z_{0})}|Dw|^{2}\right)^{1/2}\leq N\omega_{2r}(g).
\end{align}
Here $\omega_{\bullet}(g)$ denotes the modulus of continuity of $g$ in the $L_\infty$ sense.
Notice that $v:=u-w$ satisfies
\begin{align*}
\begin{cases}
-v_{t}+\Delta v=D_d\bar g^d&\mbox{in}~Q_{r,+}^{-}(z_{0}),\\
D_{d}v=\bar{g}^{d}&\mbox{on}~ \Gamma_{r}(z_{0}).
\end{cases}
\end{align*}
For any $\mathbf q=(q_{1},\dots,q_{d})=(q',q_{d})\in\mathbb R^{d}$, $\bar v:=D_{x'}v-q'$ satisfies
\begin{align*}
\begin{cases}
-\bar v_{t}+\Delta \bar v=0&\mbox{in}~Q_{r,+}^{-}(z_{0}),\\
D_{d}\bar v=0&\mbox{on}~ \Gamma_{r}(z_{0}).
\end{cases}
\end{align*}
Then by the standard parabolic estimates for equations with constant coefficients and zero conormal boundary data, we have
\begin{equation}\label{bar v holder}
[D_{x'} v]_{C^{1/2,1}(Q_{r/2,+}^{-}(z_{0}))}\leq Nr^{-1}\left(\fint_{Q_{r,+}^{-}(z_{0})}|D_{x'} v-q'|^{2}\right)^{1/2}.
\end{equation}
Now observe that $D_d v$ satisfies a heat equation in $Q_{r,+}^{-}(z_{0})$ with a constant Dirichlet boundary condition on $\Gamma_{r}(z_{0})$. Applying \cite[Lemma 4.15]{dek2}, we get
\begin{equation}\label{DDd v}
[D_{d}v]_{C^{1/2,1}(Q_{r/2,+}^{-}(z_{0}))}
\leq Nr^{-1}\left(\fint_{Q_{r,+}^{-}(z_{0})}|D_dv-q_d|^{2}\right)^{1/2}.
\end{equation}
Let $\kappa\in(0,1/2)$ be a small constant, then combining \eqref{bar v holder} and \eqref{DDd v}, we have
\begin{align}\label{difference Dv}
\left(\fint_{Q_{\kappa r,+}^{-}(z_{0})}|Dv-(Dv)_{Q_{\kappa r,+}^{-}(z_{0})}|^{2}\right)^{1/2}&\leq 2\kappa r[Dv]_{C^{1/2,1}(Q_{r/2,+}^{-}(z_{0}))}\nonumber\\
&\leq N_{0}\kappa\left(\fint_{Q_{r,+}^{-}(z_{0})}|Dv-{\bf q}|^{2}\right)^{1/2},\quad\forall~{\bf q}\in\mathbb R^d,
\end{align}
where $N_{0}=N_{0}(d,\nu,\Lambda)$. By using $u=v+w$, the triangle inequality, \eqref{Dw est}, and \eqref{difference Dv}, we obtain
\begin{align*}
&\left(\fint_{Q_{\kappa r,+}^{-}(z_{0})}|Du-(Dv)_{Q_{\kappa r,+}^{-}(z_{0})}|^{2}\right)^{1/2}\\
&\leq \left(\fint_{Q_{\kappa r,+}^{-}(z_{0})}|Dv-(Dv)_{Q_{\kappa r,+}^{-}(z_{0})}|^{2}\right)^{1/2}+\left(\fint_{Q_{\kappa r,+}^{-}(z_{0})}|Dw|^{2}\right)^{1/2}\\
&\leq N_{0}\kappa\left(\fint_{Q_{r,+}^{-}(z_{0})}|Du-\mathbf q|^{2}\right)^{1/2}+N\kappa^{-(d+2)/2}\omega_{2r}(g).
\end{align*}
Choosing ${\bf q}=(Du)_{Q_{r,+}^{-}(z_{0})}$, we reach
\begin{equation*}
\phi(z_{0},\kappa r)\leq N_{0}\kappa\phi(z_{0},r)
+N\kappa^{-(d+2)/2}\omega_{2r}(g).
\end{equation*}
The rest of the proof follows from an iteration argument. See, for example, the proof of \cite[Proposition 4.2]{dek2}. We omit the details. Therefore, we show that $Du$ is of $L_{2}$-Dini mean oscillation.
\end{proof}

\section{Appendix}

In the appendix, we prove a generalization of Lemma \ref{solvability}. We say that $w: \mathbb R^{d+1}\rightarrow[0,\infty)$ belongs to $A_{p}$ for $p\in(1,\infty)$ if
$$\sup_{Q}\frac{w(Q)}{|Q|}\left(\frac{w^{\frac{-1}{p-1}}(Q)}{|Q|}\right)^{p-1}<\infty,$$
where the supremum is taken over all parabolic cylinders $Q$ in $\mathbb R^{d+1}$. The value of the supremum is the $A_{p}$ constant of $w$, and will be denoted by $[w]_{A_p}$.

We consider the parabolic systems on $\cQ:=(-T,0)\times\cD$, where $\cD$ is a Reifenberg flat domain and the boundary $\partial\cD$ satisfies the following assumption with a parameter $\gamma_0\in (0,1/4)$ to be specified later.
\begin{assumption}[$\gamma_0$] \label{assump coeffi}
	There exists a constant $r_{0}\in(0,1]$ such that the following conditions hold.
	
	(1) In the interior of $\cD$, $A^{\alpha\beta}$ satisfy \eqref{BMO} in some coordinate system depending on $(t_{0},x_0)$ and $r$.
	
	(2) For any $x_{0}\in\partial\cD$, $t\in\mathbb R$, and $r\in(0,r_{0}]$, there is a coordinate system depending on $(t_{0},x_0)$ and $r$ such that in this new coordinate system, we have
	\begin{equation*}
	\{(y',y^{d}): x_0^{d}+\gamma_{0}r<y^{d}\}\cap B_{R}(x_0)\subset \cD\cap B_{R}(x_0)\subset\{(y',y^{d}): x_0^{d}-\gamma_{0}r<y^{d}\}\cap B_{R}(x_0)
	\end{equation*}
	and
	\begin{equation*}
	\fint_{Q_{r}^{-}(z_{0})}|A(t,x)-(A)_{Q^{-'}_{r}(z'_{0})}|\,dx\ dt\leq\gamma_{0},
	\end{equation*}
	where $(A)_{Q^{-'}_{r}(z'_{0})}=\fint_{Q^{-'}_{r}(z'_{0})}A(z',x^{d})\,dz'$.
\end{assumption}

\begin{lemma}\label{sol weight}
	Let $p\in(1,\infty)$ and $w$ be an $A_p$ weight. There exists a constant $\gamma_0\in (0,1/4)$ depending on $d$, $p$, $\nu$, $\Lambda$, and $[w]_{A_p}$
	such that, under Assumption \ref{assump coeffi}, for any $u\in \mathcal{H}_{p,w}^{1}((-T,0)\times\cD)$ satisfying
	\begin{align}\label{eq weight}
	\begin{cases}
	\mathcal{P}u-\lambda u=\Div f&\mbox{in}~\cQ,\\
	u=0&\mbox{on}~\partial_p \cQ,
	\end{cases}
	\end{align}
	where $\lambda\geq0$ and $f\in L_{p,w}(\cQ)$, we have
	\begin{equation}\label{W1p es weight}
	\|u\|_{\mathcal{H}_{p,w}^{1}(\cQ)}\leq N\|f\|_{L_{p,w}(\cQ)},
	\end{equation}
	where $N=N(n,d,p,\nu,\Lambda,[w]_{A_{p}},r_{0},T)$. Moreover, for any $f\in L_{p,w}(\cQ)$, \eqref{eq weight} admits a unique solution $u\in  \mathcal{H}_{p,w}^{1}(\cQ)$.
\end{lemma}

\begin{proof}
The case when $\lambda>\lambda_{0}$ is proved in  \cite[Section 8]{dk1} and \cite[Theorem 7.2]{dk1}, where $\lambda_{0}>0$ is a sufficiently large constant depending on $n,d,p,\nu,\Lambda,[w]_{A_{p}}$ and $r_{0}$. For $0\leq\lambda\leq\lambda_{0}$, we set
$$v:=ue^{-\lambda_{0}t}.$$
Then we have
\begin{align*}
\begin{cases}
\mathcal{P}v-(\lambda+\lambda_{0})v=e^{-\lambda_{0}t}\Div f& \mbox{in}~\cQ,\\
v=0& \mbox{on}~\partial_p \cQ.
\end{cases}
\end{align*}
By using \cite[Theorem 7.2]{dk1}, we have
\begin{equation*}
\|v\|_{\mathcal{H}_{p,w}^{1}(\cQ)}\leq N\|e^{-\lambda_{0}t}f\|_{L_{p,w}(\cQ)}\leq N\|f\|_{L_{p,w}(\cQ)},
\end{equation*}
	where $N=N(n,d,p,\nu,\Lambda,[w]_{A_{p}},r_{0},T)$. Hence, we obtain
$$\|u\|_{\mathcal{H}_{p,w}^{1}(\cQ)}=\|ve^{\lambda_{0}t}\|_{\mathcal{H}_{p,w}^{1}(\cQ)}\leq N\|f\|_{L_{p,w}(\cQ)}.$$	
The theorem  is proved.
\end{proof}


\end{document}